\documentclass[a4paper,12pt]{amsart}
\usepackage[T1]{fontenc}
\usepackage[english]{babel}
\usepackage[left=1in,right=1in,top=1.1in,bottom=1.1in]{geometry}

\allowdisplaybreaks

\usepackage{mathtools,amssymb,mathrsfs,amscd,comment}
\usepackage{times}
\usepackage[nopatch=footnote]{microtype}
\usepackage{enumitem}
\usepackage{tikz-cd}

\usepackage[dvipsnames]{xcolor}

\usepackage{thmtools}   

\usepackage[pagebackref,linktocpage]{hyperref}

\hypersetup{
  hypertexnames=false,
  pdftitle={Degree growth of rational and meromorphic self-maps},
  pdfauthor={Fei Hu},
  colorlinks=true,
  linkcolor=red,
  citecolor=MidnightBlue,
  urlcolor=MidnightBlue
}
\usepackage[nameinlink,noabbrev]{cleveref}

\newtheorem{theorem}{Theorem}[section]
\crefname{theorem}{Theorem}{Theorems}
\newtheorem{lemma}[theorem]{Lemma}
\crefname{lemma}{Lemma}{Lemmas}
\newtheorem{proposition}[theorem]{Proposition}
\crefname{proposition}{Proposition}{Propositions}

\crefname{prop}{Proposition}{Propositions}
\newtheorem{corollary}[theorem]{Corollary}
\crefname{corollary}{Corollary}{Corollaries}

\crefname{cor}{Corollary}{Corollaries}

\crefname{conjecture}{Conjecture}{Conjectures}

\crefname{conj}{Conjecture}{Conjectures}
\newtheorem*{conj*}{Conjecture}
\crefname{conj}{Conjecture}{Conjectures}

\crefname{thmx}{Theorem}{Theorems}

\theoremstyle{definition}
\newtheorem{definition}[theorem]{Definition}
\crefname{definition}{Definition}{Definitions}

\crefname{defn}{Definition}{Definitions}

\crefname{example}{Example}{Examples}

\crefname{notation}{Notation}{Notation}
\newtheorem*{notation*}{Notation}
\crefname{notation}{Notation}{Notation}
\newtheorem*{convention*}{Convention}
\crefname{convention}{Convention}{Convention}

\crefname{problem}{Problem}{Problems}
\newtheorem{question}[theorem]{Question}
\crefname{question}{Question}{Questions}

\crefname{condition}{Condition}{Conditions}

\crefname{assumption}{Assumption}{Assumptions}
\newtheorem{settinginner}{Setting}
\newenvironment{setting}[1]
  {\renewcommand{\thesettinginner}{(#1)}\begin{settinginner}}
  {\end{settinginner}}
\crefname{settinginner}{Setting}{Settings}
\Crefname{settinginner}{Setting}{Settings}

\theoremstyle{remark}

\crefname{rmk}{Remark}{Remarks}
\newtheorem*{rmk*}{Remark}
\crefname{rmk}{Remark}{Remarks}
\newtheorem{remark}[theorem]{Remark}
\crefname{remark}{Remark}{Remarks}

\crefname{fact}{Fact}{Facts}
\newtheorem{claim}[theorem]{Claim}
\crefname{claim}{Claim}{Claims}
\newtheorem*{claim*}{Claim}
\crefname{claim}{Claim}{Claims}

\crefname{step}{Step}{Steps}

\crefname{case}{Case}{Cases}

\numberwithin{equation}{section}

\DeclareMathOperator{\Char}{char}

\DeclareMathOperator{\id}{id}

\DeclareMathOperator{\Mat}{M}

\DeclareMathOperator{\Nef}{Nef}

\DeclareMathOperator{\pr}{pr}

\DeclareMathOperator{\Psef}{Psef}

\DeclareMathOperator{\tr}{tr}

\DeclareMathOperator{\vol}{vol}

\newcommand{\ddc}{\mathrm{d}\mathrm{d}^c}
\newcommand{\diff}{\mathop{}\!\mathrm{d}}

\newcommand{\bbP}{\mathbb{P}}

\newcommand{\bC}{\mathbf{C}}

\newcommand{\bk}{\mathbf{k}}

\newcommand{\bQ}{\mathbf{Q}}
\newcommand{\bR}{\mathbf{R}}

\newcommand{\bZ}{\mathbf{Z}}

\newcommand{\cC}{\mathcal{C}}
\newcommand{\cD}{\mathcal{D}}
\newcommand{\cE}{\mathcal{E}}

\newcommand{\cH}{\mathcal{H}}

\newcommand{\cK}{\mathcal{K}}

\newcommand{\cN}{\mathcal{N}}
\newcommand{\cO}{\mathcal{O}}
\newcommand{\cP}{\mathcal{P}}
\newcommand{\cR}{\mathcal{R}}

\newcommand{\cX}{\mathcal{X}}

\newcommand{\sB}{\mathscr{B}}

\newcommand{\sO}{\mathscr{O}}

\newcommand{\Frob}{\mathrm{F}}

\newcommand{\N}{\mathsf{N}}

\newcommand{\T}{\mathsf{T}}

\newcommand{\arxiv}[1]{\href{https://arxiv.org/abs/#1}{\texttt{arXiv:#1}}}

\providecommand{\MR}[1]{}
\renewcommand{\MR}[1]{\href{https://mathscinet.ams.org/mathscinet-getitem?mr=#1}{MR#1}}

\title[Degree growth of rational and meromorphic self-maps]{Degree growth of rational and meromorphic self-maps}

\author{Fei Hu}
\address{School of Mathematics, Nanjing University, Nanjing, China}
\email{fhu@nju.edu.cn}
\date{\today}

\thanks{The author was supported by NSFC Grant \#12371045.}

\begin{document}

\begin{abstract}
Let $f$ be a dominant rational self-map of a normal projective variety over an algebraically closed field of arbitrary characteristic, or a dominant meromorphic self-map of a compact K\"ahler manifold.
We prove that in every codimension, the degree sequence of $f$ differs from its dynamical-degree growth by at most a polynomial factor, with an explicit universal bound on the polynomial exponent.
In particular, subexponential but superpolynomial degree growth cannot occur.
At a strict dynamical-degree break, we prove the sharper asymptotic
\[
\deg_{i,\omega}(f^n)=C_{f,\omega,i}\,\lambda_i(f)^n+O(\eta^n),
\]
for every $\sqrt{\lambda_{i-1}(f)\lambda_{i+1}(f)}<\eta<\lambda_i(f)$.
This extends the degree-asymptotic theorem of Dang--Favre from codimension one in characteristic zero to arbitrary codimension and arbitrary characteristic, and recovers the surface theorem of Boucksom--Favre--Jonsson.
Both results rely on a uniform estimate for the temporal discrete logarithmic Laplacian of the degree sequences, obtained by applying the Ptolemy inequality to ordered Hodge--Lorentz packages on varying finite-time orbit spaces.
\end{abstract}

\maketitle


\section{Introduction}

In algebraic and complex dynamics, dynamical degrees govern the exponential rates of degree growth, but they do not by themselves determine the finer behavior of the normalized degree sequences obtained by factoring out the dominant exponential term.
For dominant rational and meromorphic maps $f$, the fundamental obstruction is that pullback need not be functorial under iteration:
\[
(f^n)^*\neq (f^*)^n.
\]
Exceptional subvarieties may enter the indeterminacy loci of later iterates and cause degree drops.
Algebraic stability was introduced by Forn{\ae}ss--Sibony \cite{FS95} to describe this phenomenon; see also \cite[D\'efinition~1.4.4]{Sibony99}.
This brings essential difficulties in studying the dynamics of rational and meromorphic maps.
It has therefore remained open whether the growth of the normalized degree sequence can be superpolynomial yet subexponential; see, e.g., \cite[Question~9.5]{HP07}, \cite[Conjecture~2]{Silverman14}, \cite[Question~3.2]{Cantat18-ICM}, and \cite[Question~11]{Silverman26}.
We prove a uniform polynomial upper bound for the normalized degree sequence in every dimension and codimension, and at a strict dynamical-degree break we obtain convergence to a positive limit with an exponential error term.

We work throughout in either of the following settings.

\begin{setting}{P}\label{setting:projective}
Let $X$ be a normal projective variety of dimension $d\geq2$ over an algebraically closed field $\bk$ of arbitrary characteristic, let $f\colon X\dashrightarrow X$ be a dominant rational self-map, and let $\omega$ be a nef and big $\bR$-Cartier divisor class on $X$.
\end{setting}

\begin{setting}{K}\label{setting:kahler}
Let $X$ be a compact K\"ahler manifold of dimension $d\geq2$, let $f\colon X\dashrightarrow X$ be a dominant meromorphic self-map, and let $\omega$ be a K\"ahler class on $X$.
\end{setting}

For $n\geq1$, let $\Gamma_n\subseteq X\times X$ be the graph closure of $f^n$.
Let $\widetilde\Gamma_n\to\Gamma_n$ be its normalization in \cref{setting:projective} and a compact K\"ahler resolution in \cref{setting:kahler}, and write $p_n,q_n\colon\widetilde\Gamma_n\to X$ for the induced projections.
For $0\leq i\leq d$, define
\begin{equation}\label{eq:intro-degree}
\deg_{i,\omega}(f^n)\coloneqq\bigl((p_n^*\omega)^{d-i}\cdot(q_n^*\omega)^i\bigr)_{\widetilde\Gamma_n},
\end{equation}
where the right-hand side is an intersection number in \cref{setting:projective} and an integral in \cref{setting:kahler}.
In the latter setting, \eqref{eq:intro-degree} is independent of the chosen resolution.
We set $\deg_{i,\omega}(f^0)\coloneqq\vol_X(\omega)=(\omega^d)_X$, the volume of the nef class $\omega$.

The quasi-submultiplicativity of degree sequences goes back to Russakovskii--Shiffman \cite{RS97} for rational maps of projective space and to Dinh--Sibony \cite{DS05a,DS04-ENS} in \cref{setting:kahler}.
Truong \cite{Truong20} and Dang \cite{Dang20} established the algebraic counterpart in \cref{setting:projective}; see also the sharper estimate of Jiang--Li \cite{JL23}.
Precisely, for each $0\leq i\leq d$, there is a constant $C_{X,\omega,i}>0$, independent of $f$, such that for all $m,n\geq1$,
\[
\deg_{i,\omega}(f^{m+n})\leq C_{X,\omega,i}\deg_{i,\omega}(f^m)\deg_{i,\omega}(f^n).
\]
So $(C_{X,\omega,i}\deg_{i,\omega}(f^n))_{n\geq1}$ is submultiplicative.
Fekete's lemma thus yields that the limit
\begin{equation}\label{eq:intro-dynamical-degree}
\lambda_i(f)\coloneqq\lim_{n\to\infty}\deg_{i,\omega}(f^n)^{1/n}=\inf_{n\geq1}\bigl(C_{X,\omega,i}\deg_{i,\omega}(f^n)\bigr)^{1/n}
\end{equation}
exists.
The quantity $\lambda_i(f)$ is called the \emph{$i$-th dynamical degree} of $f$.\footnote{In \cref{setting:kahler}, the definition \eqref{eq:intro-dynamical-degree} agrees with norm growth on $H^{i,i}$; in \cref{setting:projective}, equivalence with the $\ell$-adic cohomological definition for endomorphisms of smooth projective varieties was recently established by Xie \cite{Xie-GWRH}; cf.~\cite{Tate-WoodsHole,ES13,Truong25,Hu19,Hu24,HT26}.}
It is independent of the choice of $\omega$ and is birationally/bimeromorphically invariant.
Also, by the Khovanskii--Teissier inequality, the dynamical degrees are log-concave:
\begin{equation}\label{eq:dyn-degree-log-concave}
\lambda_i(f)^2\geq\lambda_{i-1}(f)\lambda_{i+1}(f),\qquad 1\leq i\leq d-1.
\end{equation}
From \eqref{eq:intro-dynamical-degree}, we readily obtain a uniform exponential lower bound
\begin{equation}\label{eq:standard-exponential-lower-bound}
\deg_{i,\omega}(f^n)\geq C_{X,\omega,i}^{-1}\, \lambda_i(f)^n,\qquad n\geq1.
\end{equation}

Beyond measuring degree growth, dynamical degrees determine topological entropy for holomorphic maps and bound it for rational and meromorphic maps \cite{Gromov03,Yomdin87,DS05a,DS04-ENS}, play a central role in the construction of invariant Green currents and equilibrium measures \cite{DS05-Green,Guedj05,DS10b,DDG10,DDG11}, and control the growth and equidistribution of periodic points \cite{Xie15b,DNT15,DNT17,DNV18}.
In arithmetic dynamics, they bound arithmetic degrees, enter canonical-height constructions, and provide the geometric side of the Kawaguchi--Silverman conjecture \cite{Silverman91,Kawaguchi08,KS16b,MSS18a,Matsuzawa20a,MZ22,MMSZ23,JSXZ24,MW24,HZ25,MX25}.
See the surveys \cite{DS17,BIJMST19,Silverman26} and the abundant references therein.

For a finite real sequence $a=(a_0,\ldots,a_m)$, define its \emph{discrete Laplacian} by
\[
(\Delta a)_i\coloneqq 2a_i-a_{i-1}-a_{i+1},\qquad
1\leq i\leq m-1.
\]
An index interval $[r,s]$, with $0\leq r<s\leq m$, is \emph{affine} for $a$ if $(\Delta a)_j=0$ for every integer $r<j<s$.
It is \emph{maximal affine} if it is not properly contained in any larger affine index interval.
Intervals of length one are allowed.
Applied to the concave sequence $(\log\lambda_i(f))_i$, the maximal affine intervals cover $[0,d]$ and meet only at endpoints.

Our first main theorem gives a pointwise polynomial upper bound on each maximal affine interval.

\begin{theorem}\label{thm:affine-interval-growth}
Assume that either \cref{setting:projective} or \cref{setting:kahler} holds.
Then for every maximal affine interval $[r,s]$ of $(\log\lambda_i(f))_i$, there is a constant $C_{f,\omega,r,s}>1$ such that for every $r\leq i\leq s$,
\begin{equation}\label{eq:affine-interval-growth}
\deg_{i,\omega}(f^n)\leq C_{f,\omega,r,s}\,n^{2(i-r)(s-i)}\lambda_i(f)^n,\qquad n\geq1.
\end{equation}
Consequently, there is a constant $C_{f,\omega}>1$ such that for every $0\leq i\leq d$,
\begin{equation}\label{eq:global-polynomial-growth}
\deg_{i,\omega}(f^n)\leq C_{f,\omega}\,n^{2i(d-i)}\lambda_i(f)^n,\qquad n\geq1.
\end{equation}
Together with \eqref{eq:standard-exponential-lower-bound}, this implies that for every $0\leq i\leq d$,
\begin{equation}\label{eq:global-log-growth}
\log\deg_{i,\omega}(f^n)=n\log\lambda_i(f)+O(\log n).
\end{equation}
\end{theorem}

Our \cref{thm:affine-interval-growth} simultaneously treats arbitrary dominant rational and meromorphic maps, every codimension, arbitrary characteristic, and the resonant case
\[
\lambda_i(f)^2=\lambda_{i-1}(f)\lambda_{i+1}(f).
\]
Earlier classifications and explicit bounds are available in several special situations, including birational surface maps \cite{DF01}, polynomial maps of $\bC^2$ \cite{FJ11}, monomial maps \cite{Lin12b,FW12}, zero-entropy automorphisms in codimension one \cite{LB19,DLOZ22,HJ25}, and surjective endomorphisms in any codimension \cite{HJ-degree-growth}.
Precisely, in the author's joint work with Jiang, we have proved the upper bound $2(i-r)(s-i)$ for the polynomial exponent in \cref{thm:affine-interval-growth} for surjective endomorphisms in \cref{setting:projective}; see \cite[Theorem~1.2]{HJ-degree-growth}.
The exponent $2(i-r)(s-i)$ is sharp already for isogenies of powers of elliptic curves; see \cite[Theorem~A.1(2)]{HJ-degree-growth}.
In that work, we developed a Hodge--Lorentz framework for mixed intersections and used the Ptolemy inequality to obtain upper bounds for degree growth.
In the functorial setting, these bounds can be studied through a fixed pullback operator and its Jordan decomposition.
The present paper extends the affine-interval bound to dominant rational and meromorphic maps, for which pullback need not be functorial.
We adapt the Hodge--Lorentz framework to an ordered two-cone setting on varying finite-time orbit spaces; see \cref{rmk:proof-outline}.

\begin{corollary}\label{cor:no-intermediate-growth}
Assume that either \cref{setting:projective} or \cref{setting:kahler} holds.
Suppose that $f$ has zero entropy, i.e., $\lambda_1(f)=1$.\footnote{By log-concavity of the $\lambda_i(f)$, this is equivalent to $\lambda_i(f)=1$ for some or for all $1\le i\le d-1$.}
Then for every $1\le i\le d-1$,
\begin{equation}\label{eq:zero-entropy-polynomial-bound}
\deg_{i,\omega}(f^n)=O\bigl(n^{2i(d-i)}\bigr).
\end{equation}
Thus no degree sequence in either setting can be subexponential but superpolynomial.\footnote{A positive sequence $a_n$ is subexponential if $a_n=e^{o(n)}$, superpolynomial if $\log_na_n\to+\infty$.}
\end{corollary}

\begin{remark}
In codimension one, \cref{cor:no-intermediate-growth} gives a negative answer to Hasselblatt--Propp's intermediate-growth question for dominant rational maps \cite[Question~9.5]{HP07}; see also its precursor \cite[\S5]{BV99}.
For a dominant rational self-map $f$ of $\bbP^d$, Silverman \cite[Conjecture~2]{Silverman14} conjectured that the quantity
\begin{equation}\label{eq:Silverman-ell-f}
\ell_f \coloneqq \inf\biggl\{\ell\geq0:\sup_{n\geq1} \frac{\deg_{1,\sO(1)}(f^n)}{n^\ell \lambda_1(f)^n}<\infty\biggr\}
\end{equation}
exists and is an integer with $0\leq \ell_f\leq d$.
For $X=\bbP^d$ and $i=1$, the estimate \eqref{eq:global-polynomial-growth} proves that $\ell_f$ is finite and at most $2(d-1)$, but it does not assert integrality.
Cantat observed that $\ell_f$ can already equal $4$ in dimension three through the Oguiso--Truong quotient construction \cite{OT15}; see \cite[Example~1.11]{Urech18}.
The same estimate \eqref{eq:global-polynomial-growth} proves the polynomial-upper-bound part of Cantat's question for birational maps with $\lambda_1(f)=1$ \cite[Question~3.2]{Cantat18-ICM}, but it does not establish the additional alternative that every unbounded degree sequence grows at least linearly.
Finally, \eqref{eq:global-log-growth} affirmatively answers \cite[Question~11]{Silverman26} and establishes its higher-codimensional and K\"ahler analogues.
We refer to \cref{question:polynomial-dyn-deg,rmk:polynomial-dyn-deg-known} for a further discussion.
\end{remark}

\begin{remark}
As another consequence of \cref{thm:affine-interval-growth}, we obtain an explicit exponential-polynomial upper bound for the volume of the iterated graph $\Gamma_{[n]}$:
\[
\vol(\Gamma_{[n]})=O\bigl(n^{(d^3+2d)/3}\Lambda(f)^n\bigr),\qquad \Lambda(f)\coloneqq\max_{1\leq i\leq d}\lambda_i(f),
\]
which refines Dinh--Sibony's volume estimate $\vol(\Gamma_{[n]})\leq C_\epsilon\, n^d(\Lambda(f)+\epsilon)^n$; see \cite[Lemma~5 and End of the proof of Theorem~1]{DS05a}.
We refer to \cref{cor:volume-growth} for a slightly stronger bound with $d$ replaced by the length of the interval of maximal dynamical degrees.
Furthermore, in \cite{HS-ht-growth}, we apply the present method to establish a similar height growth of rational self-maps over number fields via Yuan--Zhang's arithmetic intersection theory \cite{YZ26}.
\end{remark}

Our second main theorem concerns degree growth at a strict point of the logarithmic dynamical-degree sequence.
For $1\leq i\leq d-1$, set
\[
s_i(f)\coloneqq\sqrt{\lambda_{i-1}(f)\lambda_{i+1}(f)}.
\]
We call $i$ a \emph{strict dynamical-degree break} if
\[
\lambda_i(f)>s_i(f).
\]
This is weaker than \emph{$i$-cohomological hyperbolicity}, which requires $\lambda_i(f)>\max_{j\neq i}\lambda_j(f)$.

For a dominant meromorphic surface map, Boucksom--Favre--Jonsson \cite{BFJ08} proved, by working on the Riemann--Zariski space $\cX$ and a Hilbert space of divisor classes on $\cX$, that the strict inequality $\lambda_1(f)^2>\lambda_2(f)$ implies
\begin{equation}\label{eq:BFJ}
\deg_{1,\omega}(f^n)=C_{f,\omega,1}\,\lambda_1(f)^n+O\bigl(\lambda_2(f)^{n/2}\bigr).
\end{equation}
Cantat \cite{Cantat11} independently introduced the same Riemann--Zariski approach in his work on the Cremona group, and Xie \cite{Xie15b} later used it to prove lower semicontinuity of $\lambda_1$.
In higher-dimensional algebraic dynamics, Dang and Favre \cite{DF21} made a foundational advance by constructing Banach spaces of numerical $b$-classes in every codimension and interpreting dynamical degrees as spectral radii.
They \cite[Theorem~1]{DF21} proved the codimension-one analogue of \eqref{eq:BFJ} for normal projective varieties in characteristic zero, with remainder $O(\eta^n)$ for every $\sqrt{\lambda_2(f)}<\eta<\lambda_1(f)$.
Degree growth has also been studied in many more specialized settings; see, e.g., \cite{Furter99,BK04,BK06,BK08,BT10,BC16-JAMS,Deserti18,Urech18,LB19,CX20,LU21,DLOZ22,Dang24,Hu-GK-AV,HJ25}.
At a strict dynamical-degree break, the polynomial upper bound of \cref{thm:affine-interval-growth} can be sharpened to a pure exponential asymptotic.

\begin{theorem}\label{thm:strict-break-asymptotics}
Assume that either \cref{setting:projective} or \cref{setting:kahler} holds, and let $1\leq i\leq d-1$ be a strict dynamical-degree break for $f$.
Then there is a constant $C_{f,\omega,i}>0$ such that for every $s_i(f)<\eta<\lambda_i(f)$, one has
\begin{equation}\label{eq:main-asymptotic}
\deg_{i,\omega}(f^n)=C_{f,\omega,i}\,\lambda_i(f)^n+O(\eta^n).
\end{equation}
Moreover, the leading constant is given by
\begin{equation}\label{eq:constant-C-exp-sum}
C_{f,\omega,i}=\vol_X(\omega)\exp\!\left(\sum_{m\geq1}m\log\frac{\deg_{i,\omega}(f^m)^2}{\deg_{i,\omega}(f^{m-1})\deg_{i,\omega}(f^{m+1})}\right),
\end{equation}
where the series converges absolutely with exponentially decaying tails.
\end{theorem}

In \cref{setting:projective}, \cref{thm:strict-break-asymptotics} extends Dang--Favre's degree-asymptotic theorem \cite[Theorem~1]{DF21} from codimension one in characteristic zero to every intermediate codimension and arbitrary characteristic.
In \cref{setting:kahler}, it gives a higher-dimensional, arbitrary-codimension analogue of the Boucksom--Favre--Jonsson surface theorem \cite[Main~Theorem]{BFJ08}.
On surfaces, the two adjacent degree sequences have exact exponential rates, so that the endpoint remainder in \eqref{eq:BFJ} holds in both settings; see \cref{cor:surface-endpoint}\ref{cor:surface-endpoint-strict-break}.

\begin{remark}\label{rmk:strict-break-sharpness}
Because \eqref{eq:main-asymptotic} holds for every $\eta>s_i(f)$, it is equivalent to
\[
\limsup_{n\to\infty}\left|\deg_{i,\omega}(f^n)-C_{f,\omega,i}\,\lambda_i(f)^n\right|^{1/n}\leq s_i(f).
\]
The threshold $s_i(f)$ in the above estimate is optimal.
However, one cannot replace the remainder in \eqref{eq:main-asymptotic} by $O\bigl(s_i(f)^n\bigr)$ in general.
For every $d\geq3$ and every $1\leq i\leq d-1$, we construct an isogeny of a power of an elliptic curve for which the remainder is asymptotic to a nonzero multiple of $ns_i(f)^n$; see \cref{prop:abelian-endpoint-failure}.
\end{remark}

\begin{remark}\label{rmk:strict-break-consequences}
\Cref{thm:strict-break-asymptotics} also yields exponential convergence of successive degree ratios $\deg_{i,\omega}(f^{n+1})/\deg_{i,\omega}(f^n)$; see \cref{cor:quotient}.
It further gives a meromorphic continuation of the $i$-th degree generating series 
\[
\cD_{i,f,\omega}(z)\coloneqq \sum_{n\geq0}\deg_{i,\omega}(f^n)z^n
\]
to the disk $|z|<s_i(f)^{-1}$, with a unique simple pole at $z=\lambda_i(f)^{-1}$; see \cref{cor:generating-series}.
In the projective setting, the continuation radius $s_i(f)^{-1}$ is optimal in every dimension and intermediate codimension when $\Char(\bk)\neq2$; see \cref{rmk:BDJ-Nguyen}.
\end{remark}

Both \cref{thm:affine-interval-growth,thm:strict-break-asymptotics} follow from a uniform estimate for the temporal discrete logarithmic Laplacian of the degree sequences.

\begin{theorem}\label{thm:degree-log-laplacian}
Assume that either \cref{setting:projective} or \cref{setting:kahler} holds.
For every $1\leq i\leq d-1$, there exists a constant $C'_{f,\omega,i}>0$ such that for every $n\geq1$,
\begin{equation}\label{eq:degree-log-laplacian}
\left|\log\frac{\deg_{i,\omega}(f^n)^2}{\deg_{i,\omega}(f^{n-1})\deg_{i,\omega}(f^{n+1})}\right|
\leq C'_{f,\omega,i}\, \frac{\sqrt{\deg_{i-1,\omega}(f^n)\deg_{i+1,\omega}(f^n)}}{\deg_{i,\omega}(f^n)}.
\end{equation}
\end{theorem}

Here the codimension $i$ is the spatial variable and the iterate $n$ is the temporal variable.
Finite-time Khovanskii--Teissier inequalities yield spatial log-concavity of the degree sequences, while \eqref{eq:degree-log-laplacian} controls the temporal discrete logarithmic Laplacian by their neighboring spatial degrees.
Xie's recursive inequalities \cite{Xie-SC} provide a different kind of control on degree sequences; among other consequences, they imply that dynamical degrees are lower semicontinuous in the Zariski topology in algebraic families of dominant rational self-maps.

\begin{remark}\label{rmk:proof-outline}
We briefly outline the proofs.
For each finite orbit segment, we work over a normalized projective orbit variety or a compact K\"ahler resolution of the corresponding analytic orbit space.
These spaces are variants of the iterated graphs used in \cite{Gromov03,Friedland91,DS05a,DS08}.
The coordinate pullbacks of $\omega$ are nef and big, adjacent coordinate classes are uniformly comparable in the pseudo-effective order, and a two-time shift identity converts mixed intersections at times $m$ and $\ell$ into degrees of $f^{\ell-m}$ multiplied by $\lambda_d(f)^m$.
Together, these data form an ordered Hodge--Lorentz package in the sense of \cref{def:ordered-HL}.

The four-point Ptolemy inequality controls the determinant appearing in a mixed logarithmic derivative of the moving intersection function.
Integrating over the parameter square then yields the uniform estimate \eqref{eq:degree-log-laplacian}.
At a strict dynamical-degree break, the $n$-th root of the right-hand side of \eqref{eq:degree-log-laplacian} tends to $s_i(f)/\lambda_i(f)<1$.
The logarithmic slopes of the degree sequence are therefore exponentially Cauchy, and an elementary lemma on sequences proves \cref{thm:strict-break-asymptotics}.

On a maximal affine interval $[r,s]$, the normalized degree sequences at its endpoints are bounded by \cref{thm:strict-break-asymptotics}.
Combining the same temporal estimate \eqref{eq:degree-log-laplacian} with a discrete Phragm\'en--Lindel\"of argument \cref{lemma:two-variable-maximum-principle} gives the polynomial upper bound with exponent $2(i-r)(s-i)$ for the normalized degree sequences.
This can be regarded as a two-variable maximum principle.
\end{remark}

The paper is organized as follows.
In \S\ref{sec:HL-Ptolemy-Laplacian}, we introduce ordered Hodge--Lorentz packages and derive the abstract temporal logarithmic-Laplacian estimate \cref{thm:abstract-log-laplacian} from Ptolemy.
In \S\ref{sec:orbit-realizations}, we construct the projective and compact K\"ahler finite-time orbit realizations and prove \cref{thm:degree-log-laplacian}.
In \S\ref{sec:proof-asymptotics}, we prove first the strict-break asymptotic \cref{thm:strict-break-asymptotics} and then the affine-interval polynomial upper bound \cref{thm:affine-interval-growth}, together with \cref{cor:no-intermediate-growth}.
Finally, in \S\ref{sec:further-consequences}, we treat volume growth of iterated graphs, the successive quotients, degree generating series, endpoint behavior, and conclude with \cref{question:polynomial-dyn-deg}.

\subsection*{Acknowledgments}
The author thanks Chen Jiang for the preceding collaboration \cite{HJ-degree-growth} and for helpful discussions.
The author gratefully acknowledges the hospitality of the Fields Institute during Summer 2026, when the majority of this work was carried out.

\subsection*{Declaration on the use of generative AI}
The author used OpenAI's ChatGPT (5.5 and 5.6) during the preparation and revision of this manuscript.
Assistance from the tool was instrumental in developing \cref{thm:degree-log-laplacian}, which controls the temporal discrete logarithmic Laplacian of the degree sequence in terms of the neighboring spatial degrees.
The tool also assisted in simplifying the proof of \cref{lemma:two-variable-maximum-principle} via a discrete Phragm\'en--Lindel\"of argument.
It was used to cross-check
the explicit matrix computations in \cref{prop:abelian-endpoint-failure}.
It also assisted with revisions to the language and exposition throughout the manuscript.
The author independently reviewed and verified all mathematical statements, arguments, and computations, and takes sole responsibility for the content of the manuscript, including the accuracy of its citations.

\section{Ordered Hodge--Lorentz packages and an abstract discrete logarithmic-Laplacian estimate}
\label{sec:HL-Ptolemy-Laplacian}

The projective and compact K\"ahler arguments use the same two positivity structures: a cone of nef classes for mixed intersections, and a larger pseudo-effective cone defining an order.
We extract precisely those properties needed to derive a uniform estimate for the mixed logarithmic derivatives of moving intersection functions from the Ptolemy inequality, and prove the abstract form \cref{thm:abstract-log-laplacian} of the temporal discrete logarithmic-Laplacian estimate \cref{thm:degree-log-laplacian}.

\subsection{Ordered Hodge--Lorentz packages}

\begin{definition}\label{def:ordered-HL}
Let $V$ be a finite-dimensional real vector space, let $\cN\subseteq\cP\subseteq V$ be nonempty convex cones containing the origin, and let $Q\colon V^d\to\bR$ be a symmetric multilinear form, where $d\geq2$.
We call $(V,\cN,\cP,Q)$ an \emph{ordered Hodge--Lorentz package} of degree $d$ if the following conditions hold.
\begin{enumerate}[label=\textnormal{(\roman*)}]
\item \label{def:ordered-HL-i} Nonnegativity: For every $\xi\in\cP$ and $\eta_2,\ldots,\eta_d\in\cN$,
\[
Q(\xi,\eta_2,\ldots,\eta_d)\geq0.
\]
\item \label{def:ordered-HL-ii} Lorentzian: For every $\eta_3,\ldots,\eta_d\in\cN$, the symmetric bilinear form
\[
\sB\colon V\times V \to \bR, \quad (v_1,v_2)\longmapsto Q(v_1,v_2,\eta_3,\ldots,\eta_d)
\]
has positive index at most one.
\end{enumerate}
We write $v_1\preceq_{\cP}v_2$ if $v_2-v_1\in\cP$; this is the preorder induced by $\cP$.
\end{definition}

This definition is an ordered two-cone variant of the Hodge--Lorentz package $(V,\cC,Q)$ developed by Hu--Jiang \cite{HJ-degree-growth}.
The cone $\cN$ specifies the classes used in the mixed Hodge--Lorentz forms and the Ptolemy inequality, \cref{lemma:Ptolemy}, while the larger cone $\cP$ induces the comparison preorder.
No closedness, salience, or full-dimensionality assumption on either cone is used in the abstract argument.

Below we give two standard geometric realizations of ordered Hodge--Lorentz packages, one in the projective setting and one in the compact K\"ahler setting.
We refer to \cite[\S2]{FL17a} and \cite[\S2]{DHZ15} for brief reviews on the cones of ample/K\"ahler, nef, big, and pseudo-effective classes, in the two settings respectively.

\begin{proposition}\label{prop:geometric-HL}
The following are ordered Hodge--Lorentz packages.
\begin{enumerate}[label=\textnormal{(\alph*)}]
\item If $Y$ is a normal projective variety of dimension $d\ge2$ over an algebraically closed field, take
\[
V=\N^1(Y)_\bR,\qquad
\cN=\Nef^1(Y),\qquad
\cP=\Psef^1(Y),
\]
and
\[
Q(D_1,\ldots,D_d)=(D_1\cdots D_d)_Y.
\]
\item If $Y$ is a compact K\"ahler manifold of dimension $d\ge2$, take
\[
V=H^{1,1}(Y,\bR),\qquad
\cN=\overline{\cK}(Y),\qquad
\cP=\cE(Y),
\]
and
\[
Q(\alpha_1,\ldots,\alpha_d)=\int_Y\alpha_1\wedge\cdots\wedge\alpha_d.
\]
\end{enumerate}
In either realization, the mixed intersection of nef and big classes is strictly positive.
\end{proposition}

\begin{proof}
In the projective case, an effective $\bR$-Cartier divisor has nonnegative intersection with ample divisors.
Approximating the nef divisor classes $\eta_2,\ldots,\eta_d$ by ample classes and the pseudo-effective divisor class $\xi$ by effective classes shows that a pseudo-effective divisor has nonnegative intersection with nef divisors.
For frozen ample classes $h_3,\ldots,h_d$, the arbitrary-characteristic mixed Hodge index theorem (see \cite[Proposition~2.9]{Hu20a}; cf. \cite[Variant~1.6.2 and Remark~1.6.5]{Lazarsfeld-I}, \cite[Theorem~5.20]{YZ17}) shows that the resulting bilinear form $\sB$ is negative definite on the primitive hyperplane
\[
\{v\in V:Q(v,h,h_3,\ldots,h_d)=0\}
\]
for any ample class $h$, while $Q(h,h,h_3,\ldots,h_d)>0$.
The form therefore has positive index one.
Replacing frozen nef divisor classes $\eta_j$, with $3\le j\le d$, by $\eta_j+\epsilon h$ and letting $\epsilon\to 0^+$ proves that the positive index is at most one in the nef case.

In the compact K\"ahler case, represent a pseudo-effective class by a positive closed $(1,1)$-current and first pair it with K\"ahler forms.
Approximating each nef class by K\"ahler classes and passing to the limit verifies the nonnegativity condition~\ref{def:ordered-HL-i}.
For frozen K\"ahler classes, the mixed Hodge--Riemann theorem \cite{Gromov90,DN06} yields negative definiteness on the corresponding primitive hyperplane, and hence the positive index one assertion.
Approximating the frozen nef classes $\eta_3,\ldots,\eta_d$ by K\"ahler classes again proves that the positive index is at most one.

It remains to show strict positivity.
Fix an ample class in the projective case and a K\"ahler class in the analytic case, and denote it by $h$.
If $\omega_1,\ldots,\omega_d$ are nef and big, then for suitable $\epsilon_j>0$ the classes $\omega_j-\epsilon_jh$ are pseudo-effective.
Repeatedly applying the nonnegativity condition~\ref{def:ordered-HL-i} yields
\[
Q(\omega_1,\ldots,\omega_d)\geq\epsilon_1\cdots\epsilon_d\,Q(h,\ldots,h)>0.
\qedhere
\]
\end{proof}

\subsection{The four-point Ptolemy inequality}

We recall the following form of the four-point Ptolemy inequality from \cite{HJ-degree-growth}; see also \cite[Proposition~3.1]{HHMWW} for the equivalent zero-diagonal matrix characterization and Huh's lectures \cite{Huh-SRI} at the 2025 Summer Research Institute.
The short determinant proof is included for completeness.

\begin{lemma}[{see \cite[Theorem~2.10(2)]{HJ-degree-growth}}]\label{lemma:Ptolemy}
Let $V$ be a finite-dimensional real vector space and let $\sB\colon V\times V\to\bR$ be a symmetric bilinear form with positive index at most one.
If $v_1,\ldots,v_4\in V$ satisfy $\sB(v_i,v_j)\geq0$ for all $1\le i,j\le 4$, then
\begin{equation}\label{eq:four-point-Ptolemy}
\Bigl|\sqrt{\sB(v_1,v_2)\sB(v_3,v_4)}-\sqrt{\sB(v_1,v_4)\sB(v_2,v_3)}\Bigr|\leq \sqrt{\sB(v_1,v_3)\sB(v_2,v_4)}.
\end{equation}
\end{lemma}

\begin{proof}
Set $b_{ij}=\sB(v_i,v_j)$ and let $G=(b_{ij})_{1\leq i,j\leq4}$ be the Gram matrix.
Its positive index is at most that of $\sB$, hence at most one.
The matrix
\[
M=
\begin{pmatrix}
0&b_{12}&b_{13}&b_{14}\\
b_{12}&0&b_{23}&b_{24}\\
b_{13}&b_{23}&0&b_{34}\\
b_{14}&b_{24}&b_{34}&0
\end{pmatrix}
\]
satisfies $M=G-\mathrm{diag}(b_{11},\ldots,b_{44})$.
Since $b_{ii}\ge0$ for all $1\le i\le 4$, $M$ has positive index at most one.
If $M$ is nonsingular, its zero trace forces its signature to be $(1,3)$, and hence $\det M<0$; if it is singular, $\det M=0$.
Thus $\det M\leq0$.
Set
\[
a=\sqrt{b_{12}b_{34}},\qquad
b=\sqrt{b_{13}b_{24}},\qquad
c=\sqrt{b_{14}b_{23}}.
\]
A direct expansion gives
\[
\det M=\bigl(b^2-(a+c)^2\bigr)\bigl(b^2-(a-c)^2\bigr).
\]
Since $(a-c)^2\leq(a+c)^2$, the inequality $\det M\leq0$ is equivalent to
\[
(a-c)^2\leq b^2 \leq (a+c)^2.
\]
Thus $|a-c|\leq b \le a+c$.
The first inequality is exactly \eqref{eq:four-point-Ptolemy}.
\end{proof}

\subsection{A uniform estimate on moving intersection functions}

Let $(V,\cN,\cP,Q)$ be an ordered Hodge--Lorentz package of degree $d\geq2$.
Let $A_0,A_1,B_0,B_1\in\cN$ be arbitrary.
For $s,t\in\bR$, set
\[
A_s\coloneqq(1-s)A_0+sA_1,\qquad
B_t\coloneqq(1-t)B_0+tB_1,
\]
and for each $0\leq j\leq d$, define a two-variable polynomial function
\[
\Phi_j(s,t)\coloneqq Q\bigl(A_s^{[d-j]},B_t^{[j]}\bigr),
\]
where a superscript in square brackets denotes repetition of an argument.
By abuse of notation, we also call $\Phi_j$ the \emph{moving intersection function} of four vectors $A_0,A_1,B_0,B_1$.
The proposition below gives a uniform estimate on the mixed logarithmic derivative of $\Phi_i$ in terms of the neighboring moving intersection functions $\Phi_{i-1}$ and $\Phi_{i+1}$, under a comparison hypothesis on the four vectors.

\begin{proposition}\label{prop:Ptolemy-estimate-unit-square}
With notation as above, suppose that for some $K\geq1$,
\begin{equation}\label{eq:comparison-A0A1-B0B1}
K^{-1}A_0\preceq_{\cP}A_1\preceq_{\cP}KA_0,\qquad
K^{-1}B_0\preceq_{\cP}B_1\preceq_{\cP}KB_0.
\end{equation}
Fix $1\leq i\leq d-1$ and assume that $\Phi_i>0$ on $[0,1]^2$.
Then for every $(s,t)\in[0,1]^2$,
\begin{equation}\label{eq:uniform-Ptolemy-unit-square}
\bigl|\partial_s\partial_t\log\Phi_i(s,t)\bigr|
\leq 2i(d-i)K^4\frac{\sqrt{\Phi_{i-1}(s,t)\Phi_{i+1}(s,t)}}{\Phi_i(s,t)}.
\end{equation}
\end{proposition}

\begin{proof}
By the comparison hypotheses \eqref{eq:comparison-A0A1-B0B1} and convexity of $\cP$, we have for $s,t\in[0,1]$,
\begin{equation}\label{eq:comparison-AsA0A1-BtB0B1}
A_0,A_1\preceq_{\cP}KA_s,\qquad
B_0,B_1\preceq_{\cP}KB_t.
\end{equation}
Indeed,
\[
KA_s-A_0=(K-1)(1-s)A_0+s(KA_1-A_0)\in\cP,
\]
and the other three comparisons are identical.

Note that $\Phi_i$ is a polynomial in $s$ and $t$.
Hence positivity of $\Phi_i$ on $[0,1]^2$ implies that $\log\Phi_i$ is smooth on an open neighborhood of $[0,1]^2$.
Fix $(s,t)\in[0,1]^2$ so that $A_s,B_t\in\cN$.
Define a symmetric bilinear form on $V$ by
\[
\sB(v_1,v_2)\coloneqq Q\bigl(v_1,v_2,A_s^{[d-i-1]},B_t^{[i-1]}\bigr).
\]
The form $\sB$ has positive index at most one by \cref{def:ordered-HL}\ref{def:ordered-HL-ii}, and $\Phi_i(s,t)=\sB(A_s,B_t)$.
By multilinearity and taking partial derivatives of $\Phi_i$, we have
\[
\partial_s\Phi_i=(d-i)\sB(u,B_t),\qquad
\partial_t\Phi_i=i\sB(A_s,v),\qquad
\partial_s\partial_t\Phi_i=i(d-i)\sB(u,v),
\]
where $u\coloneqq A_1-A_0$ and $v\coloneqq B_1-B_0$.
It follows that
\begin{equation}\label{eq:log-hessian}
\partial_s\partial_t\log\Phi_i(s,t)
=i(d-i)\frac{\sB(A_s,B_t)\sB(u,v)-\sB(A_s,v)\sB(u,B_t)}{\sB(A_s,B_t)^2}.
\end{equation}
On the other hand, bilinearity of $\sB$ yields the following matrix identity:
\[
\begin{pmatrix}
\sB(A_s,B_t)&\sB(A_s,v)\\
\sB(u,B_t)&\sB(u,v)
\end{pmatrix}
=\begin{pmatrix}
1-s&s\\
-1&1
\end{pmatrix}
\begin{pmatrix}
\sB(A_0,B_0)&\sB(A_0,B_1)\\
\sB(A_1,B_0)&\sB(A_1,B_1)
\end{pmatrix}
\begin{pmatrix}
1-t&-1\\
t&1
\end{pmatrix}.
\]
Taking determinants and substituting back to \eqref{eq:log-hessian}, we obtain
\begin{equation}\label{eq:log-hessian-2}
\partial_s\partial_t\log\Phi_i(s,t)=
i(d-i)\frac{\sB(A_0,B_0)\sB(A_1,B_1)-\sB(A_0,B_1)\sB(A_1,B_0)}{\sB(A_s,B_t)^2}.
\end{equation}

All pairings among $A_0,B_0,A_1,B_1$ under $\sB$ are nonnegative by \cref{def:ordered-HL}\ref{def:ordered-HL-i}.
Applying the four-point Ptolemy inequality, \cref{lemma:Ptolemy}, to $A_0,B_0,A_1,B_1$, and then multiplying by the sum of the two square roots in the left-hand side, we get
\begin{align}\label{eq:Ptolemy-A01B01}
\begin{split}
&\qquad\qquad\bigl|
\sB(A_0,B_0)\sB(A_1,B_1)-\sB(A_0,B_1)\sB(A_1,B_0)\bigr|\\
&\leq\sqrt{\sB(A_0,A_1)\sB(B_0,B_1)}\,\Bigl(\sqrt{\sB(A_0,B_0)\sB(A_1,B_1)}+\sqrt{\sB(A_0,B_1)\sB(A_1,B_0)}\Bigr).
\end{split}
\end{align}
Use the comparison \eqref{eq:comparison-AsA0A1-BtB0B1} to bound the right-hand side of \eqref{eq:Ptolemy-A01B01} as follows:
\begin{equation}\label{eq:AsBt-comparison}
\begin{aligned}
\sB(A_0,A_1) &\le K \sB(A_s,A_1) \le K^2 \sB(A_s,A_s) = K^2 \Phi_{i-1}(s,t),\\[2pt]
\sB(B_0,B_1) &\le K \sB(B_t,B_1) \le K^2 \sB(B_t,B_t) = K^2 \Phi_{i+1}(s,t),\\[2pt]
\sB(A_j,B_k) &\le K^2 \sB(A_s,B_t) = K^2 \Phi_i(s,t),\qquad
j,k\in \{0,1\}.
\end{aligned}
\end{equation}
Substituting the estimates \eqref{eq:Ptolemy-A01B01} and \eqref{eq:AsBt-comparison} into \eqref{eq:log-hessian-2}, we prove \eqref{eq:uniform-Ptolemy-unit-square}.
\end{proof}

\begin{remark}
The estimate only requires the four classes to belong to the formal nef cone $\cN$.
In dynamical applications, these classes are also big so that mutual comparison \eqref{eq:comparison-A0A1-B0B1} in the pseudo-effective order and positivity of $\Phi_i$ on $[0,1]^2$ hold simultaneously.
\end{remark}

\subsection{An abstract estimate for the discrete logarithmic Laplacian}
\label{sec:uniformity}

The following theorem formulates the unit-square argument in a form common to both dynamical settings.
The adjacent comparison hypothesis \ref{thm:abstract-log-laplacian-comparison} controls the moving intersections uniformly, while the two-time shift hypothesis \ref{thm:abstract-log-laplacian-two-time-shift} identifies the integrated mixed derivative with the temporal discrete logarithmic Laplacian of the sequence.
Both hypotheses are satisfied by the finite-time orbit realizations constructed in \S\ref{sec:orbit-realizations}.

\begin{theorem}\label{thm:abstract-log-laplacian}
Fix integers $d,N\geq2$.
Let $(V,\cN,\cP,Q)$ be an ordered Hodge--Lorentz package of degree $d\geq2$, and let $\eta_0,\ldots,\eta_N\in\cN$.
Suppose that there exist constants $K\geq1$ and $\lambda_d>0$, and positive numbers
\[
\delta_j(m)>0,\qquad
0\leq j\leq d,\ 0\leq m\leq N,
\]
such that the following hypotheses hold.

\begin{enumerate}[label=\textnormal{(\roman*)}]
\item\label{thm:abstract-log-laplacian-comparison}
For every $0\leq m<N$,
\[
K^{-1}\eta_m\preceq_{\cP}\eta_{m+1}\preceq_{\cP} K\eta_m.
\]

\item\label{thm:abstract-log-laplacian-two-time-shift}
For every $0\leq m\leq\ell\leq N$ and $0\leq j\leq d$,
\[
Q\bigl(\eta_m^{[d-j]},\eta_\ell^{[j]}\bigr)=\lambda_d^m\delta_j(\ell-m).
\]
\end{enumerate}
Then for every $1\leq i\leq d-1$ and $1\leq n\leq N-1$,
\begin{equation}\label{eq:abstract-log-laplacian}
\left|\log\frac{\delta_i(n)^2}{\delta_i(n-1)\delta_i(n+1)}\right|
\leq 2i(d-i)K^{2d+4}\frac{\sqrt{\delta_{i-1}(n)\delta_{i+1}(n)}}{\delta_i(n)}.
\end{equation}
\end{theorem}

\begin{proof}
Fix $1\leq i\leq d-1$ and $1\leq n\leq N-1$.
We apply \cref{prop:Ptolemy-estimate-unit-square} to
\[
\eta_0,\eta_1,\eta_n,\eta_{n+1}\in\cN.
\]
For $s,t\in[0,1]$, set
\[
\eta_s\coloneqq(1-s)\eta_0+s\eta_1,\qquad
\eta_{n+t}\coloneqq(1-t)\eta_n+t\eta_{n+1},
\]
and for $0\leq j\leq d$,
\[
\Phi_j(s,t)\coloneqq Q\bigl(\eta_s^{[d-j]},\eta_{n+t}^{[j]}\bigr).
\]

The hypothesis~\ref{thm:abstract-log-laplacian-comparison}, applied at $m=0$ and $m=n$, verifies \eqref{eq:comparison-A0A1-B0B1} and implies
\[
K^{-1}\eta_0\preceq_{\cP} \eta_s \preceq_{\cP} K\eta_0,\qquad
K^{-1}\eta_n \preceq_{\cP} \eta_{n+t} \preceq_{\cP} K\eta_n.
\]
By multilinearity, the nonnegativity condition in \cref{def:ordered-HL}\ref{def:ordered-HL-i}, and the hypothesis~\ref{thm:abstract-log-laplacian-two-time-shift} here, we have for every $0\leq j\leq d$ and $(s,t)\in[0,1]^2$,
\begin{equation}\label{eq:abstract-moving-degree-comparison}
K^{-d}\delta_j(n)\leq \Phi_j(s,t)\leq K^d\delta_j(n).
\end{equation}
In particular, $\Phi_i$ is positive on $[0,1]^2$.

Applying \cref{prop:Ptolemy-estimate-unit-square} and then using \eqref{eq:abstract-moving-degree-comparison}, we obtain
\begin{equation}\label{eq:abstract-uniform-unit-square}
\begin{aligned}
\bigl|\partial_s\partial_t\log\Phi_i(s,t)\bigr|
&\leq 2i(d-i)K^4 \frac{\sqrt{\Phi_{i-1}(s,t)\Phi_{i+1}(s,t)}}{\Phi_i(s,t)}\\
&\leq 2i(d-i)K^{2d+4} \frac{\sqrt{\delta_{i-1}(n)\delta_{i+1}(n)}}{\delta_i(n)}.
\end{aligned}
\end{equation}
By hypothesis~\ref{thm:abstract-log-laplacian-two-time-shift}, the four corner values are
\[
\begin{aligned}
\Phi_i(0,0)&=\delta_i(n), &\quad\Phi_i(1,0)&=\lambda_d\delta_i(n-1),\\
\Phi_i(0,1)&=\delta_i(n+1), & \quad\Phi_i(1,1)&=\lambda_d\delta_i(n).
\end{aligned}
\]
Hence the fundamental theorem of calculus yields
\[
\begin{aligned}
\int_0^1\!\int_0^1\partial_s\partial_t\log\Phi_i(s,t)\diff s\diff t
&=\log\frac{\Phi_i(1,1)\Phi_i(0,0)}{\Phi_i(1,0)\Phi_i(0,1)}\\
&=\log\frac{\delta_i(n)^2}{\delta_i(n-1)\delta_i(n+1)}.
\end{aligned}
\]
By integrating \eqref{eq:abstract-uniform-unit-square} and the triangle inequality, we prove \eqref{eq:abstract-log-laplacian}.
\end{proof}

\section{Finite-time orbit realizations}
\label{sec:orbit-realizations}

For each finite-time orbit segment, we place all time coordinates on one geometric space.
In both settings, the coordinate pullbacks satisfy an orbit-uniform adjacent comparison and a two-time shift identity; \cref{thm:degree-log-laplacian} then follows from \cref{thm:abstract-log-laplacian}.
The projective construction uses normalization, while the compact K\"ahler construction requires the resolution of singularities.

\subsection{The projective realization}
\label{sec:degrees-orbits}

Throughout this subsection, we work over an algebraically closed field $\bk$ of arbitrary characteristic.
A variety is integral, separated, and of finite type over $\bk$.
For a divisor, it means an $\bR$-Cartier divisor.
We use standard facts from intersection theory and divisor positivity; see \cite{Fulton98,Lazarsfeld-I}.

For a normal projective variety $Y$, let
\[
\Nef^1(Y)\subseteq \Psef^1(Y)\subseteq\N^1(Y)_\bR
\]
denote the nef cone and pseudo-effective cone of divisors on $Y$.
Their interiors are the ample cone and big cone, respectively.
For divisor classes $D_1,D_2$ on $Y$, write
\[
D_1\preceq_Y D_2 \quad\Longleftrightarrow\quad D_2-D_1\in\Psef^1(Y).
\]

We will repeatedly use the following standard fact.

\begin{lemma}\label{lemma:projective-pullback-positivity}
Let $\varphi\colon Y\to Z$ be a surjective morphism of normal projective varieties.
Then pullback preserves pseudo-effective and nef divisor classes.
If, moreover, $\varphi$ is generically finite, then it also preserves nef and big divisor classes.
\end{lemma}

We work in \cref{setting:projective}.
Recall that for $n\geq1$, $\widetilde\Gamma_n$ denotes the normalization of the graph closure of $f^n$.
The two induced morphisms $p_n,q_n\colon\widetilde\Gamma_n\to X$ are surjective and generically finite, so $p_n^*\omega$ and $q_n^*\omega$ are nef and big on $\widetilde\Gamma_n$.
For $0\leq i\leq d$, the $i$-th degree of $f^n$ with respect to $\omega$ is defined by
\[
\deg_{i,\omega}(f^n)\coloneqq \bigl((p_n^*\omega)^{d-i}\cdot(q_n^*\omega)^i\bigr)_{\widetilde\Gamma_n},\qquad
\deg_{i,\omega}(f^0)\coloneqq \vol_X(\omega)=(\omega^d)_X.
\]
Once $\omega$ is fixed, we suppress $\omega$ from the notation and write
\[
\deg_i(f^n)\coloneqq\deg_{i,\omega}(f^n).
\]
By the Khovanskii--Teissier inequality \cite[Example~1.6.4 and Remark~1.6.5]{Lazarsfeld-I}, the degree sequences are log concave in the sense that for every fixed $n\geq1$,
\begin{equation}\label{eq:projective-degree-log-concavity}
\deg_i(f^n)^2\geq\deg_{i-1}(f^n)\deg_{i+1}(f^n),\qquad 1\leq i\leq d-1.
\end{equation}
Taking $n$-th roots in \eqref{eq:projective-degree-log-concavity} and the limits yields the log-concavity \eqref{eq:dyn-degree-log-concave} of dynamical degrees $\lambda_i(f)$.
Recall that by \cite{Truong20,Dang20,JL23}, the limits \eqref{eq:intro-dynamical-degree} defining the $\lambda_i(f)$ exist.
In particular, one has $\lambda_0(f)=1$ and
\[
\lambda_d(f)=\deg(f)\coloneqq [\bk(X):f^*\bk(X)].
\]

For each fixed $N\geq1$, we shall construct a normalized finite-time orbit variety $Y_N$.
Choose a dense open subset $U_N\subseteq X$ on which $f,f^2,\ldots,f^N$ are all well defined.
Let
\[
\iota_N\colon U_N\longrightarrow X^{N+1},\qquad
x\longmapsto(x,f(x),\ldots,f^N(x)),
\]
and define
\begin{equation}\label{eq:projective-finite-orbit-space}
Y_N\coloneqq\bigl(\overline{\iota_N(U_N)}\bigr)^\nu
\end{equation}
to be the normalization of the closure of the image of $\iota_N$.
The closure is unchanged after shrinking $U_N$, so $Y_N$ is independent of this choice.
It is a normal projective variety of dimension $d$, which plays a similar role as the iterated graph $\Gamma_{[N]}$ in \cite{DS08}.

For $0\leq m\leq N$, each coordinate projection\footnote{Throughout, the projection from various product spaces to its first coordinate is indexed by $0$.} from $X^{N+1}$ to $X$ induces a morphism
\[
\pi_{m,N}\colon Y_N\longrightarrow X.
\]
Note that $\pi_{0,N}$ is birational and $\pi_{m,N}=f^m\circ\pi_{0,N}$ as rational maps.
In particular, each $\pi_{m,N}$ is surjective and generically finite with degree
\[
\deg(\pi_{m,N})=\deg(f^m)=\lambda_d(f)^m.
\]
For $0\leq m\leq N$, define
\begin{equation}\label{eq:projective-omega-mN}
\omega_{m,N}\coloneqq\pi_{m,N}^*\omega\in\N^1(Y_N)_\bR.
\end{equation}
Every $\omega_{m,N}$ is nef and big on $Y_N$ by \cref{lemma:projective-pullback-positivity}.

The following two-time shift identity can be viewed as the two-time specialization of the multigraph pushforward formula used by Dinh--Sibony in the proof of \cite[Proposition~5.2, first inline formula on p.~41]{DS08}.
It verifies the condition~\ref{thm:abstract-log-laplacian-two-time-shift} of \cref{thm:abstract-log-laplacian} in the projective realization.

\begin{lemma}\label{lemma:projective-two-time-shift}
With notation as above, for every $0\leq i\leq d$ and $0\leq m\leq\ell\leq N$,
\begin{equation}\label{eq:projective-two-time-shift}
\bigl(\omega_{m,N}^{d-i}\cdot \omega_{\ell,N}^i\bigr)_{Y_N}=\lambda_d(f)^m\deg_i(f^{\ell-m}).
\end{equation}
\end{lemma}

\begin{proof}
Fix $0\leq i\leq d$ and $0\leq m\leq\ell\leq N$.
Set $k=\ell-m$.
If $m=\ell$, then the projection formula yields
\[
\bigl(\omega_{m,N}^d\bigr)_{Y_N}=\deg(\pi_{m,N})(\omega^d)_X=\lambda_d(f)^m\deg_i(f^0),
\]
so \eqref{eq:projective-two-time-shift} holds.
We thus assume that $m<\ell$.
The image of the two-coordinate morphism
\[
(\pi_{m,N},\pi_{\ell,N})\colon Y_N\longrightarrow X\times X
\]
is the graph closure $\Gamma_k$ of $f^k$.
Indeed, over a dense open subset $U_N \subseteq X$, its image consists of
\[
\bigl(f^m(x),f^\ell(x)\bigr)=\bigl(x_m,f^k(x_m)\bigr),\qquad
x_m\coloneqq f^m(x).
\]
Since $f^m$ is dominant, these points are dense in $\Gamma_k$.
We thus think of $(\pi_{m,N},\pi_{\ell,N})$ as a surjective generically finite morphism from $Y_N$ to $\Gamma_k$.
Because $Y_N$ is normal, by the universal property of normalization, the morphism $(\pi_{m,N},\pi_{\ell,N})$ factors uniquely through the normalization $\widetilde\Gamma_k$ as a surjective generically finite morphism
\[
\varphi_{m,\ell}^{(N)}\colon Y_N\longrightarrow\widetilde\Gamma_k
\]
with
\[
\pi_{m,N}=p_k\circ\varphi_{m,\ell}^{(N)},\qquad
\pi_{\ell,N}=q_k\circ\varphi_{m,\ell}^{(N)}.
\]
In other words, we have the following commutative diagram
\begin{equation}\label{eq:projective-two-time-shift-diagram}
\begin{tikzcd}[column sep=5.2em, row sep=4.2em]
& Y_N \arrow[dl, swap, "\pi_{m,N}"] \arrow[d, "\varphi_{m,\ell}^{(N)}"'] \arrow[dr, "\pi_{\ell,N}"] & \\
X & \widetilde\Gamma_k \arrow[l, "p_k"'] \arrow[r, "q_k"] & X.
\end{tikzcd}
\end{equation}

Because $p_k$ is birational, $\deg\bigl(\varphi_{m,\ell}^{(N)}\bigr)=\deg(\pi_{m,N})=\lambda_d(f)^m$.
The projection formula therefore yields
\begin{align*}
\bigl(\omega_{m,N}^{d-i}\cdot \omega_{\ell,N}^i\bigr)_{Y_N}=
\deg\bigl(\varphi_{m,\ell}^{(N)}\bigr)\bigl((p_k^*\omega)^{d-i}\cdot(q_k^*\omega)^i\bigr)_{\widetilde\Gamma_k}
=\lambda_d(f)^m\deg_i(f^k),
\end{align*}
which proves \eqref{eq:projective-two-time-shift}.
\end{proof}

The following adjacent comparison of nef and big classes verifies the condition~\ref{thm:abstract-log-laplacian-comparison} of \cref{thm:abstract-log-laplacian} in the projective realization.

\begin{lemma}\label{lemma:projective-adjacent-comparison}
With notation as above, there is a constant $K_{f,\omega}\geq1$ such that for every $N\geq1$ and $0\leq m<N$,
\begin{equation}\label{eq:projective-adjacent-comparison}
K_{f,\omega}^{-1}\,\omega_{m,N}\preceq_{Y_N}\omega_{m+1,N}\preceq_{Y_N}K_{f,\omega}\,\omega_{m,N}.
\end{equation}
\end{lemma}

\begin{proof}
On $Y_1=\widetilde\Gamma_1$, one has
\[
\pi_{0,1}=p_1,\qquad
\pi_{1,1}=q_1.
\]
The classes $\omega_{0,1}$ and $\omega_{1,1}$ are big.
Since the big cone is the interior of the pseudo-effective cone, there is $K_{f,\omega}\geq1$, depending only on $\omega_{0,1}$ and $\omega_{1,1}$, such that
\begin{equation}\label{eq:projective-adjacent-comparison-Y1}
K_{f,\omega}^{-1}\,\omega_{0,1}\preceq_{Y_1}\omega_{1,1}\preceq_{Y_1}K_{f,\omega}\,\omega_{0,1}.
\end{equation}
Fix $N\geq1$ and $0\leq m<N$.
The adjacent-coordinate projection $(\pi_{m,N},\pi_{m+1,N})$ factors uniquely through a surjective generically finite morphism
\[
\varphi_{m,m+1}^{(N)}\colon Y_N\longrightarrow Y_1=\widetilde\Gamma_1
\]
satisfying that $\pi_{m,N}=p_1\circ\varphi_{m,m+1}^{(N)}$ and $\pi_{m+1,N}=q_1\circ\varphi_{m,m+1}^{(N)}$; see also the commutative diagram \eqref{eq:projective-two-time-shift-diagram} with $\ell=m+1$.
It follows that
\[
\omega_{m,N}=\bigl(\varphi_{m,m+1}^{(N)}\bigr)^*\omega_{0,1},\qquad
\omega_{m+1,N}=\bigl(\varphi_{m,m+1}^{(N)}\bigr)^*\omega_{1,1}.
\]
So by \cref{lemma:projective-pullback-positivity}, pulling back the fixed comparison \eqref{eq:projective-adjacent-comparison-Y1} under $\varphi_{m,m+1}^{(N)}$ proves \eqref{eq:projective-adjacent-comparison}.
\end{proof}

We can now apply the abstract estimate \cref{thm:abstract-log-laplacian} to prove \cref{thm:degree-log-laplacian} in the projective setting.

\begin{proposition}\label{prop:projective-log-laplacian}
Assume that \cref{setting:projective} holds.
For every $1\leq i\leq d-1$ and $n\geq1$,
\begin{equation}\label{eq:projective-log-laplacian}
\left|\log\frac{\deg_i(f^n)^2}{\deg_i(f^{n-1})\deg_i(f^{n+1})}\right|
\leq C'_{f,\omega,i}\,\frac{\sqrt{\deg_{i-1}(f^n)\deg_{i+1}(f^n)}}{\deg_i(f^n)},
\end{equation}
where $C'_{f,\omega,i}=2i(d-i)K_{f,\omega}^{2d+4}$ and $K_{f,\omega}$ is the constant in \cref{lemma:projective-adjacent-comparison}.
\end{proposition}

\begin{proof}
Fix $1\leq i\leq d-1$ and $n\geq1$, and set $N=n+1$.
On the normal projective variety $Y_N$, take the ordered Hodge--Lorentz package
\[
V=\N^1(Y_N)_\bR,\qquad
\cN=\Nef^1(Y_N),\qquad
\cP=\Psef^1(Y_N),
\]
with
\[
Q(D_1,\ldots,D_d)=\bigl(D_1\cdots D_d\bigr)_{Y_N},
\]
as in \cref{prop:geometric-HL}.
For $0\leq m\leq N$ and $0\leq j\leq d$, set
\[
\eta_m=\omega_{m,N},\qquad
\lambda_d=\lambda_d(f)>0,\qquad
\delta_j(m)=\deg_j(f^m)>0.
\]
Each $\eta_m$ is nef and big.
The preorder induced by $\cP$ is precisely $\preceq_{Y_N}$.

By \cref{lemma:projective-adjacent-comparison}, there is a constant $K_{f,\omega}\geq1$, independent of $N$, such that
\[
K_{f,\omega}^{-1}\,\eta_m\preceq_{\cP}\eta_{m+1}\preceq_{\cP}K_{f,\omega}\,\eta_m,\qquad 0\leq m<N.
\]
Furthermore, \cref{lemma:projective-two-time-shift} yields that, for every $0\leq m\leq\ell\leq N$ and $0\leq j\leq d$,
\[
Q\bigl(\eta_m^{[d-j]},\eta_\ell^{[j]}\bigr)=\lambda_d(f)^m\deg_j(f^{\ell-m})=\lambda_d^m\delta_j(\ell-m).
\]
Thus \cref{thm:abstract-log-laplacian} proves \eqref{eq:projective-log-laplacian} with $C'_{f,\omega,i}=2i(d-i)K_{f,\omega}^{2d+4}$.
\end{proof}

\subsection{The compact K\"ahler realization}
\label{sec:kahler-dynamics}

For a compact K\"ahler manifold $Y$, let
\[
\cK(Y)\subseteq H^{1,1}(Y,\bR)
\]
be the K\"ahler cone of K\"ahler classes and let
\[
\cE(Y)\subseteq H^{1,1}(Y,\bR)
\]
be the pseudo-effective cone of classes of positive closed $(1,1)$-currents.
The closure $\overline{\cK}(Y)\subset H^{1,1}(Y,\bR)$ of $\cK(Y)$ is the nef cone and the interior of $\cE(Y)$ is the big cone.
For classes $\alpha_1,\alpha_2\in H^{1,1}(Y,\bR)$, write
\[
\alpha_1\preceq_Y \alpha_2 \quad\Longleftrightarrow\quad \alpha_2-\alpha_1\in \cE(Y).
\]
We refer to \cite{Demailly-agbook} for positivity on compact K\"ahler manifolds.

Though the following parallel lemma is now standard, we include the proof for completeness.

\begin{lemma}\label{lemma:kahler-pullback-positivity}
Let $\varphi\colon Y\to Z$ be a surjective holomorphic map between compact K\"ahler manifolds.
Then pullback preserves the pseudo-effective cone and the nef cone:
\[
\varphi^*\cE(Z)\subseteq\cE(Y),\qquad
\varphi^*\overline{\cK}(Z)\subseteq\overline{\cK}(Y).
\]
If, moreover, $\varphi$ is generically finite, then the pullback of every nef and big class on $Z$ is nef and big on $Y$.
\end{lemma}

\begin{proof}
The first assertion follows from \cite[Proposition~III.1.19 and Theorem~I.5.11]{Demailly-agbook}.
Indeed, let $\alpha\in\cE(Z)$, and choose a positive closed $(1,1)$-current $T$ representing $\alpha$, locally, of the form
\[
T=\ddc u,
\]
where $u$ is plurisubharmonic.
Since $\varphi$ is surjective, $u\circ\varphi$ is not identically $-\infty$, and the pullback current $\varphi^*T$ is locally of the form $\ddc(u\circ\varphi)$, which is a positive closed $(1,1)$-current representing $\varphi^*\alpha$.
Thus $\varphi^*\alpha$ is pseudo-effective.

Now let $\alpha\in\overline{\cK}(Z)$ be nef.
Fix K\"ahler forms $\omega_Y$ and $\omega_Z$ on $Y$ and $Z$, respectively.
For every $\epsilon>0$, there is a smooth representative $\alpha_\epsilon$ of $\alpha$ such that
\[
\alpha_\epsilon\geq-\epsilon\omega_Z.
\]
Since $Y$ is compact and $\omega_Y$ is a K\"ahler form, the form $\varphi^*\omega_Z$ is bounded above by $C\omega_Y$ for some $C>0$.
It follows that
\[
\varphi^*\alpha_\epsilon\geq -\epsilon\varphi^*\omega_Z\geq -C\epsilon\omega_Y,
\]
so $\varphi^*\alpha\in\overline{\cK}(Y)$ is nef.

Finally, suppose further that $\varphi$ is generically finite and that $\alpha$ is nef and big.
Set $d=\dim Y=\dim Z$.
Then the projection formula yields
\[
\int_Y(\varphi^*\alpha)^d
=
\deg(\varphi)\int_Z\alpha^d
>0.
\]
Since $\varphi^*\alpha$ is nef, it is big by \cite[Theorem~0.5]{DP04}.
\end{proof}

We now work in \cref{setting:kahler}.
Choose a K\"ahler form representing $\omega$ and denote the form by the same letter.
For $n\geq1$, let $\Gamma_n\subseteq X\times X$ be the graph closure of $f^n$, endowed with its reduced complex analytic structure.
By embedded analytic desingularization \cite{Hironaka77,BM97}, choose a compact K\"ahler resolution
\[
\sigma_n\colon\widetilde\Gamma_n\longrightarrow\Gamma_n.
\]
Denote by $\pr_0,\pr_1\colon X\times X\to X$ the two coordinate projections, and write
\begin{equation}\label{eq:kahler-pnqn}
p_n\coloneqq(\pr_0|_{\Gamma_n})\circ\sigma_n,\qquad
q_n\coloneqq(\pr_1|_{\Gamma_n})\circ\sigma_n
\end{equation}
for the induced holomorphic maps $\widetilde\Gamma_n\to X$.
For $0\leq i\leq d$, define the $i$-th degree of $f^n$ with respect to $\omega$ by
\[
\deg_{i,\omega}(f^n)\coloneqq \int_{\widetilde\Gamma_n}p_n^*\omega^{d-i}\wedge q_n^*\omega^i,\qquad
\deg_{i,\omega}(f^0)\coloneqq \vol_X(\omega)=\int_X \omega^d.
\]
Equivalently, denoting the current of integration over $\Gamma_n$ by $[\Gamma_n]$ and the natural pairing between currents and smooth differential forms of complementary degree by $\langle -,- \rangle$,
\begin{equation}\label{eq:kahler-degree-current}
\deg_{i,\omega}(f^n)=\Bigl\langle [\Gamma_n], \, \pr_0^*\omega^{d-i}\wedge\pr_1^*\omega^i \Bigr\rangle.
\end{equation}
Hence $\deg_{i,\omega}(f^n)$ is independent of the choice of resolution $\sigma_n$.
We henceforth fix a K\"ahler resolution $\sigma_n\colon\widetilde\Gamma_n\to\Gamma_n$ for every $n\geq1$.
Once $\omega$ is fixed, we also shorthand
\[
\deg_i(f^n)=\deg_{i,\omega}(f^n).
\]

The classes $p_n^*\omega$ and $q_n^*\omega$ are nef and big on $\widetilde\Gamma_n$, so the Khovanskii--Teissier inequality \cite{Gromov90,DN06} yields the log-concavity of the degree sequence: for every fixed $n\geq1$,
\begin{equation}\label{eq:kahler-degree-log-concavity}
\deg_i(f^n)^2\geq\deg_{i-1}(f^n)\deg_{i+1}(f^n),\qquad 1\leq i\leq d-1.
\end{equation}
It was proved by Dinh--Sibony \cite{DS05a,DS04-ENS} that the dynamical degree
\[
\lambda_i(f)\coloneqq\lim_{n\to\infty}\deg_i(f^n)^{1/n}
\]
exists and is independent of the choice of K\"ahler class $\omega$.
In particular, one has $\lambda_0(f)=1$ and $\lambda_d(f)=\deg(f)$, the generic topological degree of $f$.

For every fixed $N\geq1$, choose a dense Zariski-open subset $U_N\subseteq X$ on which $f,f^2,\ldots,f^N$ are holomorphic, and set
\[
\iota_N\colon U_N\longrightarrow X^{N+1},\qquad
x\longmapsto\bigl(x,f(x),\ldots,f^N(x)\bigr).
\]
Let
\[
Z_N\coloneqq \overline{\iota_N(U_N)}\subseteq X^{N+1}
\]
be the finite-time orbit closure (this is also called the iterated graph $\Gamma_{[N]}$ in \cite{DS08}).
It is an irreducible compact analytic space of dimension $d$.
By embedded analytic desingularization \cite{Hironaka77,BM97}, choose a compact K\"ahler resolution 
\begin{equation}\label{eq:kahler-finite-orbit-space}
\nu_N\colon Y_N\longrightarrow Z_N.
\end{equation}
For $0\leq m\leq N$, denote the coordinate projection by $\pr_m\colon X^{N+1}\to X$, define
\begin{equation}\label{eq:kahler-pi-mN}
\pi_{m,N}\coloneqq \bigl(\pr_m|_{Z_N}\bigr)\circ\nu_N\colon Y_N\longrightarrow X
\end{equation}
and
\begin{equation}\label{eq:kahler-omega-mN}
\omega_{m,N}\coloneqq \pi_{m,N}^*\omega\in H^{1,1}(Y_N,\bR).
\end{equation}

Note that $\pi_{0,N}$ is bimeromorphic and $\pi_{m,N}=f^m\circ\pi_{0,N}$ as meromorphic maps.
In particular, each $\pi_{m,N}$ is surjective and generically finite with generic degree
\begin{equation}\label{eq:kahler-coordinate-degree}
\deg(\pi_{m,N})=\deg(f^m)=\lambda_d(f)^m.
\end{equation}
By \cref{lemma:kahler-pullback-positivity}, every $\omega_{m,N}$ is nef and big.

We now verify the two-time shift identity and adjacent comparison for the classes $\omega_{m,N}$ on $Y_N$.

\begin{lemma}\label{lemma:kahler-two-time-shift}
With notation as above, for every $0\leq i\leq d$ and $0\leq m\leq\ell\leq N$,
\begin{equation}\label{eq:kahler-two-time-shift}
\int_{Y_N}\omega_{m,N}^{d-i}\wedge\omega_{\ell,N}^i=\lambda_d(f)^m\deg_i(f^{\ell-m}).
\end{equation}
\end{lemma}

\begin{proof}
If $m=\ell$, then by the projection formula and \eqref{eq:kahler-coordinate-degree},
\[
\int_{Y_N}\omega_{m,N}^d=\deg(\pi_{m,N})\int_X\omega^d=\lambda_d(f)^m\deg_i(f^0),
\]
which proves \eqref{eq:kahler-two-time-shift}.
We thus assume that $m<\ell$.
Set $k\coloneqq\ell-m$.
Similarly, the two-coordinate map
\[
\varphi_{m,\ell}^{(N)}\coloneqq(\pi_{m,N},\pi_{\ell,N})\colon Y_N\longrightarrow X\times X
\]
has image equal to $\Gamma_k$.
Indeed, on the dense regular-orbit locus,
\[
\varphi_{m,\ell}^{(N)}(x)=\bigl(f^m(x),f^\ell(x)\bigr)=\bigl(x_m,f^k(x_m)\bigr),\qquad
x_m\coloneqq f^m(x),
\]
and these points are dense in $\Gamma_k$ because $f^m$ is dominant.
We therefore regard $\varphi_{m,\ell}^{(N)}$ as a surjective holomorphic map
\[
\varphi_{m,\ell}^{(N)}\colon Y_N\longrightarrow\Gamma_k.
\]
Moreover, by the definition of $\varphi_{m,\ell}^{(N)}$,
\begin{equation}\label{eq:kahler-pi-varphi}
\pi_{m,N}=(\pr_0|_{\Gamma_k})\circ\varphi_{m,\ell}^{(N)},\qquad
\pi_{\ell,N}=(\pr_1|_{\Gamma_k})\circ\varphi_{m,\ell}^{(N)}.
\end{equation}

Since $\pr_0|_{\Gamma_k} \colon \Gamma_k \to X$ is bimeromorphic, we have
\[
\deg\bigl(\varphi_{m,\ell}^{(N)}\bigr)=\deg(\pi_{m,N})=\lambda_d(f)^m,
\]
where the last equality is \eqref{eq:kahler-coordinate-degree}.
It follows that
\[
\bigl(\varphi_{m,\ell}^{(N)}\bigr)_*[Y_N]=\lambda_d(f)^m[\Gamma_k].
\]
Therefore, by the projection formula and \eqref{eq:kahler-degree-current}, we obtain
\begin{align*}
\int_{Y_N}\omega_{m,N}^{d-i}\wedge\omega_{\ell,N}^i
&= \Bigl\langle\bigl(\varphi_{m,\ell}^{(N)}\bigr)_*[Y_N], \, \pr_0^*\omega^{d-i}\wedge\pr_1^*\omega^i \Bigr\rangle\\
&= \lambda_d(f)^m \Bigl\langle [\Gamma_k], \, \pr_0^*\omega^{d-i}\wedge\pr_1^*\omega^i \Bigr\rangle\\
&= \lambda_d(f)^m\deg_i(f^k).\qedhere
\end{align*}
\end{proof}

\begin{lemma}\label{lemma:kahler-adjacent-comparison}
With notation as above, there is a constant $K_{f,\omega}\geq1$ such that for every $N\geq1$ and $0\leq m<N$,
\begin{equation}\label{eq:kahler-adjacent-comparison}
K_{f,\omega}^{-1}\,\omega_{m,N}\preceq_{Y_N}\omega_{m+1,N}\preceq_{Y_N}K_{f,\omega}\,\omega_{m,N}.
\end{equation}
\end{lemma}

\begin{proof}
The classes $p_1^*\omega$ and $q_1^*\omega$ are nef and big on the fixed compact K\"ahler manifold $\widetilde\Gamma_1$.
Since the big cone is the interior of the pseudo-effective cone, there is a constant $K_{f,\omega}\geq1$, depending only on $p_1^*\omega$ and $q_1^*\omega$, such that
\begin{equation}\label{eq:kahler-fixed-graph-comparison}
K_{f,\omega}^{-1}\,p_1^*\omega\preceq_{\widetilde\Gamma_1}q_1^*\omega\preceq_{\widetilde\Gamma_1}K_{f,\omega}\,p_1^*\omega.
\end{equation}

Fix $N\geq1$ and $0\leq m<N$.
The adjacent-coordinate map
\[
\varphi_{m,m+1}^{(N)}\coloneqq(\pi_{m,N},\pi_{m+1,N})\colon Y_N\longrightarrow X\times X
\]
has image equal to $\Gamma_1$, and we regard it as a surjective holomorphic map $Y_N\to\Gamma_1$.

Recall that $\sigma_1\colon\widetilde\Gamma_1\to\Gamma_1$ is bimeromorphic, so that we have a meromorphic lift
\[
\sigma_1^{-1}\circ\varphi_{m,m+1}^{(N)}\colon Y_N\dashrightarrow \widetilde{\Gamma}_1.
\]
After resolving its indeterminacy, we obtain a compact K\"ahler manifold $W_{m,N}$, a bimeromorphic morphism $\tau_{m,N}\colon W_{m,N}\to Y_N$, and a surjective holomorphic map $\psi_{m,N}\colon W_{m,N}\to\widetilde\Gamma_1$, such that the following diagram commutes:
\[
\begin{tikzcd}
W_{m,N} \arrow[r, "\psi_{m,N}"] \arrow[d, "\tau_{m,N}"'] & \widetilde{\Gamma}_1 \arrow[d, "\sigma_1"]\\
Y_N \arrow[r, "\varphi_{m,m+1}^{(N)}"'] & \Gamma_1.
\end{tikzcd}
\]
Moreover, by \eqref{eq:kahler-pnqn} and \eqref{eq:kahler-pi-varphi}, we have
\begin{align*}
p_1\circ\psi_{m,N}&=(\pr_0|_{\Gamma_1})\circ\sigma_1\circ\psi_{m,N}=(\pr_0|_{\Gamma_1})\circ\varphi_{m,m+1}^{(N)}\circ\tau_{m,N}=\pi_{m,N}\circ\tau_{m,N},\\
q_1\circ\psi_{m,N}&=(\pr_1|_{\Gamma_1})\circ\sigma_1\circ\psi_{m,N}=(\pr_1|_{\Gamma_1})\circ\varphi_{m,m+1}^{(N)}\circ\tau_{m,N}=\pi_{m+1,N}\circ\tau_{m,N}.
\end{align*}
It follows that
\[
\tau_{m,N}^*\omega_{m,N}=\psi_{m,N}^*p_1^*\omega,\qquad
\tau_{m,N}^*\omega_{m+1,N}=\psi_{m,N}^*q_1^*\omega.
\]
Thus pulling back the comparison \eqref{eq:kahler-fixed-graph-comparison} by $\psi_{m,N}$ and applying \cref{lemma:kahler-pullback-positivity} yield
\begin{equation}\label{eq:kahler-comparison-on-WmN}
K_{f,\omega}^{-1}\,\tau_{m,N}^*\omega_{m,N}\preceq_{W_{m,N}}\tau_{m,N}^*\omega_{m+1,N}\preceq_{W_{m,N}}K_{f,\omega}\,\tau_{m,N}^*\omega_{m,N}.
\end{equation}
Since $\tau_{m,N}$ is bimeromorphic and holomorphic, $(\tau_{m,N})_*\tau_{m,N}^*=\id$ on $H^{1,1}(Y_N,\bR)$.
By taking pushforward $(\tau_{m,N})_*$ in \eqref{eq:kahler-comparison-on-WmN}, we prove \eqref{eq:kahler-adjacent-comparison}.
\end{proof}

\begin{proposition}\label{prop:kahler-log-laplacian}
Assume that \cref{setting:kahler} holds.
For every $1\leq i\leq d-1$ and $n\geq1$,
\begin{equation}\label{eq:kahler-log-laplacian}
\left|\log\frac{\deg_i(f^n)^2}{\deg_i(f^{n-1})\deg_i(f^{n+1})}\right|
\leq C'_{f,\omega,i}\,\frac{\sqrt{\deg_{i-1}(f^n)\deg_{i+1}(f^n)}}{\deg_i(f^n)},
\end{equation}
where $C'_{f,\omega,i}=2i(d-i)K_{f,\omega}^{2d+4}$ and $K_{f,\omega}$ is the constant in \cref{lemma:kahler-adjacent-comparison}.
\end{proposition}

\begin{proof}
Fix $1\leq i\leq d-1$ and $n\geq1$, and set $N=n+1$.
On the compact K\"ahler manifold $Y_N$, take the ordered Hodge--Lorentz package
\[
V=H^{1,1}(Y_{N},\bR),\qquad
\cN=\overline{\cK}(Y_{N}),\qquad
\cP=\cE(Y_{N}),
\]
with
\[
Q(\alpha_1,\ldots,\alpha_d)=\int_{Y_{N}}\alpha_1\wedge\cdots\wedge\alpha_d,
\]
as in \cref{prop:geometric-HL}.
For $0\leq m\leq N$ and $0\leq j\leq d$, set
\[
\eta_m=\omega_{m,N},\qquad
\lambda_d=\lambda_d(f)>0,\qquad
\delta_j(m)=\deg_j(f^m)>0.
\]
The classes $\eta_m$ are nef and big.

By \cref{lemma:kahler-adjacent-comparison}, there is a constant $K_{f,\omega}\geq1$, independent of $N$, such that
\[
K_{f,\omega}^{-1}\,\eta_m\preceq_{\cP}\eta_{m+1}\preceq_{\cP}K_{f,\omega}\,\eta_m,\qquad 0\leq m<N.
\]
Moreover, \cref{lemma:kahler-two-time-shift} gives, for every $0\leq m\leq\ell\leq N$ and $0\leq j\leq d$,
\[
Q\bigl(\eta_m^{[d-j]},\eta_\ell^{[j]}\bigr)=\lambda_d(f)^m\deg_j(f^{\ell-m})=\lambda_d^m\delta_j(\ell-m).
\]
Thus both hypotheses of \cref{thm:abstract-log-laplacian} are satisfied, and we conclude that \eqref{eq:kahler-log-laplacian} holds with $C'_{f,\omega,i}=2i(d-i)K_{f,\omega}^{2d+4}$.
\end{proof}

\begin{proof}[Proof of \cref{thm:degree-log-laplacian}]
It follows from \cref{prop:projective-log-laplacian,prop:kahler-log-laplacian} that the degree sequences satisfy the temporal discrete logarithmic-Laplacian estimate \eqref{eq:degree-log-laplacian} in both settings.
\end{proof}

\section{Degree growth and asymptotics}
\label{sec:proof-asymptotics}

Throughout this section, denote the $i$-th degree $\deg_{i,\omega}(f^n)$ by $\deg_i(f^n)$ for brevity.

\subsection{Proof of Theorem~\ref{thm:strict-break-asymptotics}}

We first prove an elementary lemma on sequences that turns the temporal discrete logarithmic-Laplacian estimate on degree sequences into a precise asymptotic.

\begin{lemma}\label{lemma:discrete-integration}
Let $(a_n)_{n\geq0}$ be a sequence of positive real numbers such that
\[
\lim_{n\to\infty}a_n^{1/n}=\alpha>0.
\]
Suppose that for some $0<q<1$,
\begin{equation}\label{eq:abstract-kappa-bound}
\left|\log\frac{a_n^2}{a_{n-1}a_{n+1}}\right|=O(q^n).
\end{equation}
Then
\[
a_n=C\alpha^n\bigl(1+O(q^n)\bigr),
\]
where
\[
C=a_0\exp\!\left(\sum_{m\geq1}m \log\frac{a_m^2}{a_{m-1}a_{m+1}}\right)
=\lim_{n\to\infty}\frac{a_n}{\alpha^n}
=\lim_{n\to\infty}\frac{a_n^{n+1}}{a_{n+1}^n},
\]
and the series converges absolutely.
\end{lemma}

\begin{proof}
Denote for all appropriate $n$,
\[
b_n\coloneqq\log a_n,\qquad
c_n\coloneqq b_n-b_{n-1}=\log\frac{a_n}{a_{n-1}},\qquad
\kappa_n\coloneqq c_n-c_{n+1}=\log\frac{a_n^2}{a_{n-1}a_{n+1}}.
\]
Then the assumed \eqref{eq:abstract-kappa-bound} says that $|\kappa_n|=O(q^n)$.
Hence $(c_n)_{n\geq1}$ is an exponentially Cauchy sequence and converges to some $c\in\bR$, with
\[
c_n-c=\sum_{m\geq n}\kappa_m=O(q^n).
\]
By the Stolz--Ces\`aro theorem and the root-limit assumption, we have
\[
c=\lim_{n\to\infty}c_n
=\lim_{n\to\infty}\frac{b_n}{n}
=\lim_{n\to\infty}\log(a_n^{1/n})
=\log\alpha.
\]
Moreover, since $\sum_{m\geq1}m\kappa_m$ converges absolutely,
\[
\sum_{n\geq1}(c_n-\log\alpha)=\sum_{n\geq1}\sum_{m\geq n}\kappa_m=\sum_{m\geq1}m\kappa_m.
\]
Define
\[
C\coloneqq a_0\exp\!\left(\sum_{m\geq1}m\kappa_m\right)\in\bR_{>0}.
\]
We thus obtain
\[
\log\frac{a_n}{C\alpha^n}=b_n-b_0-n\log\alpha-\sum_{m\geq1}m\kappa_m=-\sum_{m>n}(c_m-\log\alpha)=O(q^n).
\]
Taking exponentiation, we get
\[
a_n=C\alpha^n\exp\bigl(O(q^n)\bigr)=C\alpha^n\bigl(1+O(q^n)\bigr).
\]
This asymptotic also yields $a_n/\alpha^n\to C$.
Finally, direct telescoping gives, for $n\geq1$,
\[
a_0\exp\!\left(\sum_{m=1}^nm\kappa_m\right)=\frac{a_n^{n+1}}{a_{n+1}^n}.
\]
By letting $n\to\infty$, we prove the remaining formula for $C$.
\end{proof}

\begin{proof}[Proof of \cref{thm:strict-break-asymptotics}]
Fix $\eta$ with $s_i(f)<\eta<\lambda_i(f)$ and set $q\coloneqq\eta/\lambda_i(f)$.
Recall that in either setting, the limits defining dynamical degrees exist, i.e., for every $0\leq j\leq d$,
\[
\lim_{n\to\infty}\deg_j(f^n)^{1/n}=\lambda_j(f);
\]
see \cite{DS05a,DS04-ENS} and \cite{Truong20,Dang20,JL23} in \cref{setting:kahler} and \cref{setting:projective}, respectively.
Therefore, one has
\[
\lim_{n\to\infty}\Biggl(\frac{\sqrt{\deg_{i-1}(f^n)\deg_{i+1}(f^n)}}{\deg_i(f^n)}\Biggr)^{1/n}
=\frac{s_i(f)}{\lambda_i(f)}
<\frac{\eta}{\lambda_i(f)}=q,
\]
and hence
\[
\frac{\sqrt{\deg_{i-1}(f^n)\deg_{i+1}(f^n)}}{\deg_i(f^n)}=O(q^n).
\]
The temporal discrete logarithmic-Laplacian estimate \cref{thm:degree-log-laplacian} now implies
\[
\left|\log\frac{\deg_i(f^n)^2}{\deg_i(f^{n-1})\deg_i(f^{n+1})}\right|=O(q^n).
\]
Applying \cref{lemma:discrete-integration} with $a_n=\deg_i(f^n)$ and $\alpha=\lambda_i(f)$ yields
\[
\deg_i(f^n)=C_{f,\omega,i}\,\lambda_i(f)^n\bigl(1+O(q^n)\bigr)
=C_{f,\omega,i}\,\lambda_i(f)^n+O(\eta^n),
\]
where the constant $C_{f,\omega,i}>0$ is exactly given by \eqref{eq:constant-C-exp-sum}.
Its series converges absolutely because the $m$-th summand is $O(mq^m)$.
\end{proof}

\begin{remark}\label{rmk:C-f-omega-i-limits}
By \cref{lemma:discrete-integration}, the leading constant also satisfies
\[
C_{f,\omega,i}=\lim_{n\to\infty}\frac{\deg_i(f^n)}{\lambda_i(f)^n}
=\lim_{n\to\infty}\frac{\deg_i(f^n)^{n+1}}{\deg_i(f^{n+1})^n}.
\]
\end{remark}

\subsection{Proofs of Theorem~\ref{thm:affine-interval-growth} and Corollary~\ref{cor:no-intermediate-growth}}

The properties satisfied by the normalized degree sequences $\deg_i(f^n)/\lambda_i(f)^n$ and a discrete Phragm\'en--Lindel\"of argument yield the desired upper bound.
We formulate this as an abstract lemma.

For an array $a=(a_{i,n})$, define its spatial and temporal discrete Laplacian, whenever the neighboring terms are defined, by
\begin{align*}
(\Delta_{\mathrm{sp}} a)_{i,n}&\coloneqq 2a_{i,n}-a_{i-1,n}-a_{i+1,n},\\
(\Delta_{\mathrm{tm}} a)_{i,n}&\coloneqq 2a_{i,n}-a_{i,n-1}-a_{i,n+1}.
\end{align*}

\begin{lemma}\label{lemma:two-variable-maximum-principle}
Fix two integers $r<s$.
For each $r\leq i\leq s$, let $(u_{i,n})_{n\geq0}$ be a real sequence satisfying the following properties.
\begin{enumerate}[label=\textnormal{(\roman*)}]
\item \label{lemma:two-variable-maximum-principle-infinitesimal}
For every $r\le i\le s$,
\begin{equation}\label{eq:affine-interval-root-limit}
\lim_{n\to\infty} \frac{u_{i,n}}{n}=0.
\end{equation}

\item \label{lemma:two-variable-maximum-principle-temporal-estimate}
There exists $C_0>0$ such that for every $r<j<s$ and $n\geq1$,
\begin{equation}\label{eq:affine-interval-temporal-estimate}
\bigl|(\Delta_{\mathrm{tm}} u)_{j,n}\bigr|\leq
C_0\exp\Bigl(-\frac{1}{2}(\Delta_{\mathrm{sp}} u)_{j,n}\Bigr).
\end{equation}

\item \label{lemma:two-variable-maximum-principle-bounded}
There exists $C_1>0$ such that for every $n\geq0$,
\begin{equation}\label{eq:affine-interval-bounded}
\max\{u_{r,n},u_{s,n}\}\leq C_1.
\end{equation}
\end{enumerate}
Then there is a constant $C'_{u,r,s}>0$ such that for every $r\leq i\leq s$ and $n\geq1$,
\begin{equation}\label{eq:affine-interval-conclusion}
u_{i,n}\leq C'_{u,r,s}+2(i-r)(s-i)\log n.
\end{equation}
\end{lemma}

\begin{proof}
If $s-r=1$, then the conclusion follows immediately from \eqref{eq:affine-interval-bounded} with $C'_{u,r,s}=C_1$.
We therefore assume that $s-r\geq2$.

Denote $b_i\coloneqq 2(i-r)(s-i)$ and choose $C_2,C_3\ge1$ such that
\begin{align}
C_2 &\geq \max\Bigl\{C_1, \, \max_{r\le i\le s}\{u_{i,1}\}\Bigr\} \label{eq:C2}, \\
C_0C_3^{-2}&<\min_{r<j<s}b_j=2(s-r-1). \label{eq:C3}
\end{align}
We shall show that it suffices to take $\displaystyle C'_{u,r,s}\geq C_2 + \frac{1}{2}(s-r)^2\log C_3$.

For $r\le i\le s$ and $n\geq1$, define
\[
v_{i,n}\coloneqq C_2+b_i\log(C_3n).
\]
By \eqref{eq:affine-interval-bounded} and the choice \eqref{eq:C2} of $C_2$, $u\le v$ on the boundary of the discrete half-strip $\{r,\ldots,s\}\times\bZ_{n\geq1}$, i.e., $u_{i,n}\le v_{i,n}$ for $i=r,s$, or $n=1$.

Fix an arbitrary auxiliary $\epsilon>0$.
For $r\le i\le s$ and $n\geq1$, set
\[
w_{i,n}\coloneqq u_{i,n}-v_{i,n}-\epsilon n.
\]
Then according to \eqref{eq:affine-interval-root-limit}, $\displaystyle\lim_{n\to\infty}w_{i,n}=-\infty$ uniformly for all $r\le i\le s$.
If $w$ were positive somewhere, it would attain a positive global maximum at some $r<i_\circ<s$ and $n_\circ\geq2$.
It follows that $(\Delta_{\mathrm{sp}} w)_{i_\circ,n_\circ}\geq0$ and hence
\[
(\Delta_{\mathrm{sp}} u)_{i_\circ,n_\circ}\geq (\Delta_{\mathrm{sp}} v)_{i_\circ,n_\circ}=4\log(C_3n_\circ).
\]
Combining this with \eqref{eq:affine-interval-temporal-estimate} and \eqref{eq:C3} yields that
\[
\bigl|(\Delta_{\mathrm{tm}} u)_{i_\circ,n_\circ}\bigr|
\leq C_0\exp\Bigl(-\frac{1}{2}(\Delta_{\mathrm{sp}} u)_{i_\circ,n_\circ}\Bigr)
\leq \frac{C_0 C_3^{-2}}{n_\circ^2}
<\frac{b_{i_\circ}}{n_\circ^2}.
\]
On the other hand, by $(\Delta_{\mathrm{tm}} w)_{i_\circ,n_\circ}\geq0$ and the fact that $\log(1+x)\le x$ for $x>-1$, one has
\[
(\Delta_{\mathrm{tm}} u)_{i_\circ,n_\circ}\geq(\Delta_{\mathrm{tm}} v)_{i_\circ,n_\circ}=b_{i_\circ}\log\frac{n_\circ^2}{n_\circ^2-1}=-b_{i_\circ}\log\Bigl(1-\frac{1}{n_\circ^2}\Bigr)\geq\frac{b_{i_\circ}}{n_\circ^2},
\]
a contradiction.
It follows that $w\leq0$, i.e., for every $r\le i\le s$ and $n\geq1$,
\[
u_{i,n}\leq v_{i,n}+\epsilon n.
\]
Since both $u$ and $v$ are independent of $\epsilon$, by letting $\epsilon\to0^+$, we obtain
\[
u_{i,n}\leq v_{i,n}=C_2+b_i\log C_3+b_i\log n\leq C'_{u,r,s}+b_i\log n.\qedhere
\]
\end{proof}

\begin{proof}[Proof of \cref{thm:affine-interval-growth}]
Fix a maximal affine interval $[r,s]$ of $(\log\lambda_i(f))_i$, and set
\[
u_{i,n}\coloneqq\log\frac{\deg_i(f^n)}{\lambda_i(f)^n},\qquad
r\leq i\leq s,\quad n\geq0.
\]
By the definition \eqref{eq:intro-dynamical-degree} of the dynamical degrees $\lambda_i(f)$, for every $r\le i\le s$,
\[
\lim_{n\to\infty} \frac{u_{i,n}}{n}=0.
\]

For every $r<j<s$, if any, the affineness of $[r,s]$ yields $\lambda_j(f)^2=\lambda_{j-1}(f)\lambda_{j+1}(f)$.
Hence by \cref{thm:degree-log-laplacian}, there is $C'_{f,\omega,j}>0$ such that for every $n\geq1$,
\[
\bigl|(\Delta_{\mathrm{tm}} u)_{j,n}\bigr|
\leq C'_{f,\omega,j}\, \frac{\sqrt{\deg_{j-1}(f^n)\deg_{j+1}(f^n)}}{\deg_j(f^n)}
=C'_{f,\omega,j}\exp\Bigl(-\frac{1}{2}(\Delta_{\mathrm{sp}} u)_{j,n}\Bigr).
\]
Taking the maximum of the finitely many $C'_{f,\omega,j}$, we verify \eqref{eq:affine-interval-temporal-estimate} for $u$.

It remains to verify boundedness of $u$ at the boundary indices $i=r$ and $i=s$.
If $r>0$, then maximality of $[r,s]$, together with the log-concavity of the dynamical degrees, implies that $r$ is a strict dynamical-degree break, so \cref{thm:strict-break-asymptotics} has shown that $(u_{r,n})_{n\geq0}$ is bounded (see also \cref{rmk:C-f-omega-i-limits}).
If $r=0$, then $u_{0,n}=\log\vol_X(\omega)$ is constant.
Similarly, $(u_{s,n})_n$ is bounded: if $s<d$, then $s$ is a strict break, while if $s=d$, then $u_{d,n}=\log\vol_X(\omega)$ is also constant.

Therefore, by \cref{lemma:two-variable-maximum-principle}, there is a constant $C'_{u,r,s}>0$ such that for every $r\leq i\leq s$ and $n\geq1$,
\[
u_{i,n}\leq C'_{u,r,s}+2(i-r)(s-i)\log n,
\]
or equivalently,
\[
\deg_i(f^n)\leq e^{C'_{u,r,s}}\, n^{2(i-r)(s-i)}\lambda_i(f)^n.
\]
Hence \eqref{eq:affine-interval-growth} holds with $C_{f,\omega,r,s}\coloneqq e^{C'_{u,r,s}}$.

Finally, note that every index $0\leq i\leq d$ lies in some maximal affine interval $[r,s]$ and $(i-r)(s-i)\leq i(d-i)$.
So by taking $C_{f,\omega}$ as the maximum of the finitely many $C_{f,\omega,r,s}$, the global estimate \eqref{eq:global-polynomial-growth} follows.
Combining \eqref{eq:global-polynomial-growth} with \eqref{eq:standard-exponential-lower-bound} and taking logarithms, we obtain \eqref{eq:global-log-growth}.
This concludes the proof of \cref{thm:affine-interval-growth}.
\end{proof}

\begin{proof}[Proof of \cref{cor:no-intermediate-growth}]
Since $\lambda_i(f)=1$, the upper bound in \eqref{eq:global-polynomial-growth} reduces to \eqref{eq:zero-entropy-polynomial-bound}.
By \eqref{eq:intro-dynamical-degree}, subexponential $i$-th degree growth is equivalent to $\lambda_i(f)=1$.
The same estimate therefore excludes subexponential but superpolynomial degree growth.
\end{proof}

\section{Further consequences and sharpness}
\label{sec:further-consequences}

We give an application of \cref{thm:affine-interval-growth} to volume growth of iterated graphs $\Gamma_{[n]}$ (cf.~\cite{Gromov03,DS05a,CPR21}).
We also discuss some further consequences of \cref{thm:strict-break-asymptotics}, and its sharpness and endpoint behavior.

\subsection{Volume growth of iterated graphs}\label{subsec:volume-growth}

Assume that either \cref{setting:projective} or \cref{setting:kahler} holds.
Recall from \S\ref{sec:orbit-realizations} that $Y_n$ denotes the normalization, respectively a compact K\"ahler resolution, of the iterated graph $\Gamma_{[n]}\subseteq X^{n+1}$, that is the closure of the image of
\[
x\longmapsto (x,f(x),\ldots,f^n(x))\in X^{n+1}.
\]
Let $\omega_{m,n}$ be the pullback of $\omega$ under the induced $m$-th coordinate projection $\pi_{m,n}\colon Y_n\to X$, indexed by $0\leq m\leq n$.
Define the \emph{volume} of $\Gamma_{[n]}$ by
\[
\vol(\Gamma_{[n]})\coloneqq \frac{1}{d!}\left(\left(\sum_{m=0}^n\omega_{m,n}\right)^{\!d}\right)_{Y_n}.
\]

\begin{corollary}\label{cor:volume-growth}
Assume that either \cref{setting:projective} or \cref{setting:kahler} holds.
Denote
\[
\Lambda(f)\coloneqq\max_{1\leq i\leq d}\lambda_i(f),\qquad I_{\max}\coloneqq\{i:\lambda_i(f)=\Lambda(f)\}=[r,s]\cap\bZ,\qquad \ell\coloneqq s-r.
\]
Then there are constants $c_{\vol},C_{\vol}>0$ such that for every $n\geq1$,
\begin{equation}\label{eq:volume-growth}
c_{\vol}\, \Lambda(f)^n\leq \vol(\Gamma_{[n]})\leq C_{\vol}\, n^{(\ell^3+2\ell)/3}\Lambda(f)^n.
\end{equation}
In particular, if $f$ is cohomologically hyperbolic, then $\vol(\Gamma_{[n]})\asymp\Lambda(f)^n$; if $f$ has zero entropy, then $\vol(\Gamma_{[n]})=O\bigl(n^{(d^3+2d)/3}\bigr)$.
\end{corollary}

\begin{proof}
Fix $n\geq 1$ and write $\omega_m\coloneqq\omega_{m,n}$ for $0\leq m\leq n$.
We first claim that for every ordered $d$-tuple $0\leq m_d\leq\dots\leq m_1\leq n$, one has
\begin{equation}\label{eq:volume-gap}
\bigl(\omega_{m_1}\cdots\omega_{m_d}\bigr)_{Y_n}
\leq \frac{\prod_{j=1}^{d-1}2^{j(d-j)}}{(\omega^d)_X^{d}}\prod_{i=0}^{d}\deg_{i}(f^{k_i}),
\end{equation}
where $k_0=n-m_1$, $k_{j}=m_j-m_{j+1}$ for $1\leq j<d$, and $k_d=m_d$ (cf.~\cite[Equation~(3.3)]{CPR21}).
Indeed, for $1\leq j\leq d-1$, applying the reverse Khovanskii--Teissier inequality \cite[Theorem~2.5]{HX24} (cf.~\cite[Remark~3.1]{Xiao15}, \cite[Theorem~5.7]{LX17}, \cite[Theorem~3.4.3]{Dang20}, \cite[Theorem~3.5]{JL23})
\[
\bigl(A^{d}\bigr)_{Y_n}\bigl(B_1\cdots B_d\bigr)_{Y_n}\le 2^{j(d-j)} \bigl(A^{d-j} B_1\cdots B_j\bigr)_{Y_n} \bigl(A^j B_{j+1}\cdots B_d\bigr)_{Y_n}
\]
with $A=\omega_{m_{j+1}}$, $B_1=\dots=B_j=\omega_{m_j}$, and $B_{j+1}=\omega_{m_{j+1}},\ldots,B_d=\omega_{m_d}$, we get
\begin{equation}\label{eq:Xiao-rKT}
\bigl(\omega_{m_{j+1}}^d\bigr)_{Y_n}\bigl(\omega_{m_j}^j\omega_{m_{j+1}}\cdots\omega_{m_d}\bigr)_{Y_n}
\leq 2^{j(d-j)}\bigl(\omega_{m_j}^j\omega_{m_{j+1}}^{d-j}\bigr)_{Y_n}\bigl(\omega_{m_{j+1}}^{j+1}\omega_{m_{j+2}}\cdots\omega_{m_d}\bigr)_{Y_n}.
\end{equation}
On the other hand, the two-time shift identity (see \cref{lemma:projective-two-time-shift,lemma:kahler-two-time-shift}) yields
\[
\bigl(\omega_{m_{j+1}}^d\bigr)_{Y_n}=(\omega^d)_X\lambda_d(f)^{m_{j+1}},\qquad \bigl(\omega_{m_j}^j\cdot\omega_{m_{j+1}}^{d-j}\bigr)_{Y_n}=\lambda_d(f)^{m_{j+1}}\deg_j(f^{k_j}).
\]
Iterating the preceding inequality \eqref{eq:Xiao-rKT} for $j=1,\ldots,d-1$ thus implies
\[
\bigl(\omega_{m_1}\cdots\omega_{m_d}\bigr)_{Y_n}\leq
\frac{\prod_{j=1}^{d-1}2^{j(d-j)}}{(\omega^d)_X^{d-1}}\bigl(\omega_{m_d}^d\bigr)_{Y_n}\prod_{j=1}^{d-1}\deg_j(f^{k_j}).
\]
Since $\deg_0(f^{k_0})=(\omega^d)_X$ and $\bigl(\omega_{m_d}^d\bigr)_{Y_n}=\deg_d(f^{m_d})=\deg_d(f^{k_d})$, this proves \eqref{eq:volume-gap}.

Now expand $\vol(\Gamma_{[n]})$ and sort its indices as $0\leq m_d\leq\dots\leq m_1\leq n$, with a cost at most $d!$.
By \eqref{eq:volume-gap}, there exists a constant $C_1>0$, independent of $n$, such that
\begin{equation}\label{eq:volume-convolution}
\vol(\Gamma_{[n]})\leq C_1\sum_{\substack{k_0,\ldots,k_d\geq0\\[2pt]k_0+\cdots+k_d=n}}
\prod_{i=0}^d\deg_i(f^{k_i}).
\end{equation}
By \cref{thm:affine-interval-growth} and the definitions of $\Lambda(f)$ and $I_{\max}$, there are constants $C_2>0$ and $q\in(0,1)$ such that, for every $k\geq1$,
\[
\frac{\deg_i(f^k)}{\Lambda(f)^k}\leq
\begin{cases}
C_2k^{2(i-r)(s-i)},&i\in I_{\max},\\
C_2q^k,&i\in \{0,\ldots,d\}\setminus I_{\max}.
\end{cases}
\]
Hence only the indices within $I_{\max}$ contribute to the growth of $\vol(\Gamma_{[n]})/\Lambda(f)^n$.
After summing over the indices $i\in I_{\max}$, it is bounded above by a polynomial in $n$ of degree at most
\[
\ell+\sum_{i=r}^{s}2{(i-r)}(s-i)=\ell+\sum_{j=0}^{\ell}2j(\ell-j)=\frac{\ell^3+2\ell}{3}.
\]
This proves the upper bound in \eqref{eq:volume-growth}.
Finally, by \eqref{eq:standard-exponential-lower-bound}, we have that
\[
\vol(\Gamma_{[n]})\geq \frac{\binom{d}{r}}{d!}\bigl(\omega_{0}^{d-r}\cdot\omega_{n}^{r}\bigr)_{Y_n}=\frac{\deg_r(f^n)}{r!(d-r)!}\geq \frac{\Lambda(f)^n}{r!(d-r)!C_{X,\omega,r}},
\]
so the lower bound in \eqref{eq:volume-growth} follows.
\end{proof}

\begin{remark}
We \emph{expect} that a more delicate analysis of mixed intersections could improve the upper bound to $\vol(\Gamma_{[n]})=O\bigl(n^{\ell^2}\Lambda(f)^n\bigr)$.
Since $\lambda_1(f)=1$ implies $\ell=d$, whereas $\lambda_1(f)>1$ implies $\ell\leq d-1$, this would yield
\begin{equation}
\vol(\Gamma_{[n]})=
\begin{cases}
O\bigl(n^{d^2}\bigr), & \lambda_1(f)=1,\\
O\bigl(n^{(d-1)^2}\Lambda(f)^n\bigr), & \lambda_1(f)>1.
\end{cases}
\end{equation}
Note that when $f$ is a zero-entropy automorphism, the upper bound $\vol(\Gamma_{[n]})=O\bigl(n^{d^2}\bigr)$ is already known \cite[Corollary~1.2]{HJ25}.
\end{remark}

\subsection{Quotients and degree generating series at strict breaks}

\begin{corollary}\label{cor:quotient}
Under the hypotheses of \cref{thm:strict-break-asymptotics}, for every
\[
\frac{s_i(f)}{\lambda_i(f)}<q<1,
\]
one has
\[
\frac{\deg_{i,\omega}(f^{n+1})}{\deg_{i,\omega}(f^n)}=\lambda_i(f)\bigl(1+O(q^n)\bigr).
\]
\end{corollary}

\begin{proof}
Applying \cref{thm:strict-break-asymptotics} with $\eta=q\lambda_i(f)$, we obtain
\[
\deg_{i,\omega}(f^n)=C_{f,\omega,i}\,\lambda_i(f)^n\bigl(1+O(q^n)\bigr).
\]
So, by taking successive quotients, we prove the corollary.
\end{proof}

\begin{corollary}\label{cor:generating-series}
Under the hypotheses of \cref{thm:strict-break-asymptotics}, the $i$-th degree generating series
\[
\cD_{i,f,\omega}(z)\coloneqq\sum_{n\geq0}\deg_{i,\omega}(f^n)z^n
\]
admits a meromorphic continuation to the disk
\[
|z|<s_i(f)^{-1}=\bigl(\lambda_{i-1}(f)\lambda_{i+1}(f)\bigr)^{-1/2}.
\]
More precisely,
\begin{equation}\label{eq:generating-series}
\cD_{i,f,\omega}(z)=\frac{C_{f,\omega,i}}{1-\lambda_i(f)z}+\cR_{i,f,\omega}(z),
\end{equation}
where $\cR_{i,f,\omega}$ is holomorphic on $|z|<s_i(f)^{-1}$.
In particular, $z=\lambda_i(f)^{-1}$ is the unique simple pole in the disk $|z|<s_i(f)^{-1}$, with residue $-C_{f,\omega,i}/\lambda_i(f)$.
\end{corollary}

\begin{proof}
Denote
\[
\cR_{i,f,\omega}(z)\coloneqq\sum_{n\geq0}\,\bigl(\deg_{i,\omega}(f^n)-C_{f,\omega,i}\,\lambda_i(f)^n\bigr)z^n.
\]
Since \cref{thm:strict-break-asymptotics} holds for every $\eta>s_i(f)$,
\[
\limsup_{n\to\infty}\left|\deg_{i,\omega}(f^n)-C_{f,\omega,i}\,\lambda_i(f)^n\right|^{1/n}\leq s_i(f).
\]
Hence $\cR_{i,f,\omega}$ is holomorphic on the disk $|z|<s_i(f)^{-1}$.
On the original disk of convergence $|z|<\lambda_i(f)^{-1}$ for $\cD_{i,f,\omega}(z)$, we have
\[
\cD_{i,f,\omega}(z)=C_{f,\omega,i}\sum_{n\geq0}\lambda_i(f)^nz^n+\cR_{i,f,\omega}(z)=\frac{C_{f,\omega,i}}{1-\lambda_i(f)z}+\cR_{i,f,\omega}(z),
\]
which shows \eqref{eq:generating-series}.
The right-hand side provides the claimed continuation and pole.
\end{proof}

\begin{remark}[{cf.~\cite{BDJ20,Nguyen26}}]\label{rmk:BDJ-Nguyen}
When the characteristic of the ground field $\bk$ is not $2$, the meromorphic continuation radius $s_i(f)^{-1}$ in \cref{cor:generating-series} is optimal in every dimension $d\ge2$ and every intermediate codimension.
Indeed, set $d_1\coloneqq 2$ and choose a Bell--Diller--Jonsson surface map $f\colon \bbP^{d_1}\dashrightarrow \bbP^{d_1}$ (e.g., take $\zeta=-3+4\sqrt{-1}$ there), so that
\[
\lambda_1(f) \in \bigl(\sqrt{\lambda_2(f)}, \lambda_2(f)\bigr)\setminus \overline{\bQ}.
\]
Thanks to \cite[Corollary~B]{Nguyen26}, the first-degree generating series $\cD_{1,f,\cO(1)}$ has $|z|=s_1(f)^{-1}=\lambda_2(f)^{-1/2}$ as a natural boundary.

Fix $d\ge2$ and $1\le i\le d-1$.
Set $d_2\coloneqq i-1$ and $d_3\coloneqq d-i-1$.
Choose an integer $q>\lambda_2(f)$.
Let $g_q\colon\bbP^{d_2}\to\bbP^{d_2}$ be a $q$-polarized endomorphism and define
\[
F\coloneqq f\times g_q\times\id_{\bbP^{d_3}} \colon \bbP^{d_1}\times\bbP^{d_2}\times\bbP^{d_3} \dashrightarrow \bbP^{d_1}\times\bbP^{d_2}\times\bbP^{d_3}.
\]
Using the natural product polarization $H$ on $\bbP^{d_1}\times\bbP^{d_2}\times\bbP^{d_3}$, one can check that
\begin{equation}\label{eq:general-product-degree}
\deg_{k,H}(F^n)=
\sum_{\substack{k_1+k_2+k_3=k\\[3pt]
0\leq k_j\leq d_j}}
\binom{k}{k_1,k_2,k_3}\binom{d-k}{k'_1,k'_2,k'_3}\deg_{k_1,\cO(1)}(f^n)q^{k_2n}
\end{equation}
for every $1\le k\le d-1$, where $k'_j\coloneqq d_j-k_j$ for $j=1,2,3$.
It follows that
\[
\lambda_{i-1}(F)=q^{i-1},\qquad
\lambda_i(F)=q^{i-1}\lambda_1(f),\qquad
\lambda_{i+1}(F)=q^{i-1}\lambda_2(f),
\]
and hence
\[
s_i(F)=\sqrt{\lambda_{i-1}(F)\lambda_{i+1}(F)}=q^{i-1}\sqrt{\lambda_2(f)} < \lambda_i(F).
\]
Namely, $i$ is a strict dynamical-degree break for $F$, so \cref{cor:generating-series} applies to $F$ and $i$.

Now specializing \eqref{eq:general-product-degree} to $k=i$, the dominant term of $\deg_{i,H}(F^n)$, corresponding to $(k_1,k_2,k_3)=(1,i-1,0)$, is
\[
i(d-i)q^{(i-1)n}\deg_{1,\cO(1)}(f^n).
\]
Let $R_n$ be the sum of all remaining terms, where each one is $O\bigl(q^{(i-1)n}\bigr)$; in particular,
\[
\limsup_{n\to\infty}R_n^{1/n}\le q^{i-1}<q^{i-1}\sqrt{\lambda_2(f)}=s_i(F).
\]
It follows that the $i$-th degree generating series of $F$ with respect to $H$ has the form
\[
\cD_{i,F,H}(z)=i(d-i)\cD_{1,f,\cO(1)}(q^{i-1}z)+\cH(z),
\]
where $\cH(z)\coloneqq\sum_{n\geq0}R_nz^n$ is holomorphic on a disk strictly larger than $|z|<s_i(F)^{-1}$.
Therefore, the circle $|z|=s_i(F)^{-1}$ is a natural boundary of $\cD_{i,F,H}(z)$.
\end{remark}

\subsection{Sharpness and endpoint behavior}

The examples in the next proposition show that, unlike the surface case, in dimension at least three, the remainder $O(\eta^n)$ in \cref{thm:strict-break-asymptotics} cannot generally be improved to the endpoint $O\bigl(s_i(f)^n\bigr)$.

\begin{proposition}\label{prop:abelian-endpoint-failure}
Let $E$ be an elliptic curve over $\bk$, let $d\geq3$, and let
\[
\Theta\coloneqq \sum_{j=1}^d\pr_j^*[0]
\]
be the product principal polarization on $E^d$, where $\pr_j\colon E^d\to E$ denotes the $j$-th projection.
For every $1\leq i\leq d-1$, there is an isogeny $f=f_{d,i}\colon E^d\to E^d$ such that
\[
\lambda_{i-1}(f)=9^{i-1},\qquad
\lambda_i(f)=9^i,\qquad
\lambda_{i+1}(f)=4\cdot9^i.
\]
In particular, $s_i(f)=6\cdot9^{i-1}<\lambda_i(f)$.
Moreover, for some $c_{d,i}>0$,
\begin{equation}\label{eq:abelian-endpoint-asymptotic}
\deg_{i,\Theta}(f^n)=3i!(d-i)!\lambda_i(f)^n-c_{d,i}\,n s_i(f)^n+O\bigl(s_i(f)^n\bigr).
\end{equation}
Consequently, the constant $C_{f,\Theta,i}$ in \cref{thm:strict-break-asymptotics} equals $3i!(d-i)!$, and
\[
\deg_{i,\Theta}(f^n)-C_{f,\Theta,i}\,\lambda_i(f)^n\sim -c_{d,i}\,n s_i(f)^n.
\]
\end{proposition}

\begin{proof}
Let
\[
B_1\coloneqq
\begin{pmatrix}
3&1&0\\
0&2&1\\
0&0&2
\end{pmatrix},
\qquad
B_2\coloneqq
\begin{pmatrix}
3&1&0\\
0&3&1\\
0&0&2
\end{pmatrix},
\]
and let $f=f_{d,i}$ be the isogeny of $E^d$ induced by
\[
M_{d,i}\coloneqq
\begin{cases}
B_1\oplus I_{d-3},&i=1,\\[2mm]
3I_{i-2}\oplus B_2\oplus I_{d-i-1},&2\leq i\leq d-1.
\end{cases}
\]

\begin{claim}\label{claim:degrees-matrix-norm}
For any $M\in \Mat_d(\bZ)$ with nonzero determinant, the induced isogeny $f_M$ of $E^d$ satisfies that for every $0\le j\le d$,
\begin{equation}\label{eq:abelian-matrix-degree}
\deg_{j,\Theta}(f_M^n)=j!(d-j)!\left\|\bigwedge\nolimits^jM^n\right\|_{\Frob}^2,
\end{equation}
where $\|\cdot\|_{\Frob}$ denotes the Frobenius norm of matrices.
In particular,
\[
\lambda_j(f_M)=\rho\Bigl(\bigwedge\nolimits^jM\Bigr)^2.
\]
\end{claim}

\begin{proof}[Proof of Claim \ref{claim:degrees-matrix-norm}]
Under the product principal polarization $\Theta$, $(f_M^n)^*\Theta$ corresponds to the symmetric matrix $(M^n)^{\T}M^n$.
Therefore, for every integer $t\geq0$, the polarization $\Theta+t (f_M^n)^*\Theta$ corresponds to
\[
I_d+t(M^n)^{\T}M^n.
\]
By the Riemann--Roch theorem for abelian varieties \cite[\S16]{Mumford} and the fact that $\deg(f_{M'})=(\det M')^2$ for any $M'\in\Mat_d(\bZ)$, we obtain
\[
\frac{1}{d!}\Bigl(\Theta+t\,(f_M^n)^*\Theta\Bigr)^d=\det\Bigl(I_d+t\,(M^n)^{\T}M^n\Bigr).
\]
Expanding both sides yields
\[
\sum_{j=0}^d\frac{\deg_{j,\Theta}(f_M^n)}{j!(d-j)!}\,t^j=\sum_{j=0}^d\tr\Bigl(\bigwedge\nolimits^j\bigl((M^n)^{\T}M^n\bigr)\Bigr)\,t^j.
\]
By functoriality of exterior powers,
\[
\bigwedge\nolimits^j\bigl((M^n)^{\T}M^n\bigr)=\bigwedge\nolimits^j (M^n)^{\T}\bigwedge\nolimits^j M^n=\Bigl(\bigwedge\nolimits^j M^n\Bigr)^{\T}\bigwedge\nolimits^j M^n,
\]
its trace is $\bigl\|\bigwedge^j M^n\bigr\|_{\Frob}^2$.
Comparing coefficients thus proves \eqref{eq:abelian-matrix-degree}.
The argument is algebraic and remains valid in arbitrary characteristic; see also \cite{Hu24,Hu-GK-AV}.
\end{proof}

We continue with the proof of \cref{prop:abelian-endpoint-failure}.
For every $1\le i\le d-1$, applying \cref{claim:degrees-matrix-norm} with $M=M_{d,i}$, whose $i+1$ eigenvalues of largest modulus are $\underbrace{3,\ldots,3}_{i\text{ copies}},2$, we have
\[
\lambda_{i-1}(f)=9^{i-1},\qquad
\lambda_i(f)=9^i,\qquad
\lambda_{i+1}(f)=4\cdot9^i,
\]
and hence $s_i(f)=6\cdot9^{i-1}<\lambda_i(f)$.

For $i=1$, a direct computation gives
\[
B_1^n=
\begin{pmatrix}
3^n&3^n-2^n&3^n-(n+2)2^{n-1}\\
0&2^n&n2^{n-1}\\
0&0&2^n
\end{pmatrix}.
\]
Consequently,
\[
\bigl\|B_1^n\bigr\|_{\Frob}^2=3\cdot9^n-(n+4)6^n+\frac{n^2+2n+8}{2}\,4^n.
\]
The identity block contributes only $d-3$, so
\[
\bigl\|M_{d,1}^n\bigr\|_{\Frob}^2=3\cdot9^n-n6^n+O(6^n).
\]

Now assume that $i\geq2$.
The second compound matrix of $B_2$ is
\[
\bigwedge\nolimits^2B_2=
\begin{pmatrix}
9&3&1\\
0&6&2\\
0&0&6
\end{pmatrix},
\]
and hence
\[
\bigwedge\nolimits^2B_2^n=
\begin{pmatrix}
9^n&9^n-6^n&9^n-(n/3+1)6^n\\
0&6^n&(n/3)6^n\\
0&0&6^n
\end{pmatrix}.
\]
It follows that
\begin{align*}
\left\|\bigwedge\nolimits^2B_2^n\right\|_{\Frob}^2=3\cdot81^n-\left(\frac{2}{3}n+4\right)54^n+\left(\frac{2}{9}n^2+\frac{2}{3}n+4\right)36^n.
\end{align*}
Under the orthogonal decomposition
\[
\bigwedge\nolimits^i\bigl(\bR^{i-2}\oplus\bR^3\oplus\bR^{d-i-1}\bigr)=
\bigoplus_{i_1+i_2+i_3=i}\bigwedge\nolimits^{i_1}\bR^{i-2}\otimes\bigwedge\nolimits^{i_2}\bR^3\otimes\bigwedge\nolimits^{i_3}\bR^{d-i-1},
\]
the summand in $\bigwedge\nolimits^iM_{d,i}^n$ corresponding to $(i_1,i_2,i_3)=(i-2,2,0)$ contributes
\[
9^{(i-2)n}\left\|\bigwedge\nolimits^2B_2^n\right\|_{\Frob}^2
\]
to its squared Frobenius norm.
Every other summand has squared spectral radius at most $4\cdot9^{i-1}$, which is strictly smaller than $s_i(f)=6\cdot9^{i-1}$.
The polynomial factors arising from Jordan blocks are therefore absorbed by $O\bigl(s_i(f)^n\bigr)$.
Hence
\[
\left\|\bigwedge\nolimits^iM_{d,i}^n\right\|_{\Frob}^2=3(9^i)^n-\frac{2}{3}n\bigl(6\cdot9^{i-1}\bigr)^n+O\bigl(s_i(f)^n\bigr).
\]

Finally, applying \cref{claim:degrees-matrix-norm}, we prove \eqref{eq:abelian-endpoint-asymptotic} with $C_{f,\Theta,i}=3i!(d-i)!$ and
\[
c_{d,i}=
\begin{cases}
(d-1)!,&i=1,\\[1mm]
\displaystyle\frac{2}{3}i!(d-i)!,&i\geq2.
\end{cases}
\]
The remaining assertions of \cref{prop:abelian-endpoint-failure} follow readily.
\end{proof}

The endpoint estimate does hold when the adjacent degree sequences have their pure exponential growth.
In particular, \cref{cor:surface-endpoint}\ref{cor:surface-endpoint-strict-break} recovers \cite[Main Theorem]{BFJ08}.

\begin{lemma}\label{lemma:endpoint-criterion}
Under the hypotheses of \cref{thm:strict-break-asymptotics}, suppose in addition that
\begin{equation}\label{eq:adjacent-growth-condition}
\sqrt{\deg_{i-1,\omega}(f^n)\deg_{i+1,\omega}(f^n)}=O\bigl(s_i(f)^n\bigr).
\end{equation}
Then
\[
\deg_{i,\omega}(f^n)=C_{f,\omega,i}\,\lambda_i(f)^n+O\bigl(s_i(f)^n\bigr).
\]
\end{lemma}

\begin{proof}
By \cref{thm:strict-break-asymptotics},
\[
\lim_{n\to\infty}\frac{\deg_{i,\omega}(f^n)}{\lambda_i(f)^n} = C_{f,\omega,i}>0.
\]
Let $q\coloneqq s_i(f)/\lambda_i(f)<1$.
Then the assumed \eqref{eq:adjacent-growth-condition} becomes
\[
\frac{\sqrt{\deg_{i-1,\omega}(f^n)\deg_{i+1,\omega}(f^n)}}{\deg_{i,\omega}(f^n)}=O(q^n).
\]
Substituting it into \eqref{eq:degree-log-laplacian} in \cref{thm:degree-log-laplacian} implies
\[
\left|\log\frac{\deg_{i,\omega}(f^n)^2}{\deg_{i,\omega}(f^{n-1})\deg_{i,\omega}(f^{n+1})}\right|=O(q^n).
\]
Now applying \cref{lemma:discrete-integration} to $a_n=\deg_{i,\omega}(f^n)$ and $\alpha=\lambda_i(f)$ yields
\[
\deg_{i,\omega}(f^n)=C_{f,\omega,i}\,\lambda_i(f)^n\bigl(1+O(q^n)\bigr)=C_{f,\omega,i}\,\lambda_i(f)^n+O\bigl(s_i(f)^n\bigr).\qedhere
\]
\end{proof}

\begin{corollary}[{cf.~\cite[Main Theorem]{BFJ08}}]\label{cor:surface-endpoint}
Assume that either \cref{setting:projective} or \cref{setting:kahler} holds, and $d=2$.
\begin{enumerate}[label=\textnormal{(\arabic*)}]
\item \label{cor:surface-endpoint-strict-break}
If $\lambda_1(f)^2>\lambda_2(f)$, then
\[
\deg_{1,\omega}(f^n)=C_{f,\omega,1}\,\lambda_1(f)^n+O\bigl(\lambda_2(f)^{n/2}\bigr).
\]
\item \label{cor:surface-endpoint-resonant}
If $\lambda_1(f)^2=\lambda_2(f)$, then there is $C>0$ such that for every $n\geq1$,
\[
\vol_X(\omega)\lambda_1(f)^n\leq\deg_{1,\omega}(f^n)\leq Cn^2\lambda_1(f)^n.
\]
\end{enumerate}
\end{corollary}

\begin{proof}
Note that for every $n\geq1$, $\deg_{0,\omega}(f^n)=\vol_X(\omega)$ and
\[
\deg_{2,\omega}(f^n)=\vol_X(\omega)\lambda_2(f)^n.
\]
It follows that $s_1(f)=\sqrt{\lambda_2(f)}$ and
\[
\sqrt{\deg_{0,\omega}(f^n)\deg_{2,\omega}(f^n)}=\vol_X(\omega)s_1(f)^n.
\]

If $\lambda_1(f)^2>\lambda_2(f)$, then $i=1$ is a strict dynamical-degree break for $f$.
Part~(1) thus follows from \cref{lemma:endpoint-criterion}.
Now assume $\lambda_1(f)^2=\lambda_2(f)$, i.e., $s_1(f)=\lambda_1(f)$.
Then $[0,2]$ is a maximal affine interval, so the upper bound in part~(2) follows from \cref{thm:affine-interval-growth}.
On the other hand, the Khovanskii--Teissier inequality implies
\[
\deg_{1,\omega}(f^n)^2\geq\deg_{0,\omega}(f^n)\deg_{2,\omega}(f^n)=\vol_X(\omega)^2\lambda_1(f)^{2n},
\]
for every $n\geq1$, which proves the stated lower bound.
\end{proof}

We conclude with a multipart question concerning the finer growth of degree sequences.
Following the limsup definition of Fan--Fu--Ouchi \cite[Definition~5.1]{FFO21}, define, for $0\leq i\leq d$, the \emph{upper polynomial dynamical degree}
\begin{equation}\label{eq:upper-mu}
\overline\mu_i(f)\coloneqq\limsup_{n\to\infty}\frac{\log\deg_{i,\omega}(f^n)-n\log\lambda_i(f)}{\log n}.
\end{equation}
The corresponding \emph{lower polynomial dynamical degree} $\underline\mu_i(f)$ is defined by replacing $\limsup$ with $\liminf$.
The comparison of degrees shows that both quantities are independent of the choice of $\omega$.
See \cite[Definition~1.11]{LB19} and \cite[\S3.1]{CPR21} for zero-entropy versions.
When $X=\bbP^d$ and $\omega=\sO(1)$, the invariant $\overline\mu_1(f)$ agrees with Silverman's exponent $\ell_f$; see \eqref{eq:Silverman-ell-f}.

If $f$ is regular in \cref{setting:projective}, or holomorphic in \cref{setting:kahler}, then pullback is functorial.
The comparison of degrees with the norm of $(f^*)^n$, together with Jordan normal form, shows that the limit exists, $\mu_i(f)\in\bZ_{\geq0}$, and
\[
\deg_{i,\omega}(f^n)\asymp n^{\mu_i(f)}\lambda_i(f)^n;
\]
see \cite[Lemma~4.2 and Proposition~5.3]{FFO21}.

More generally, for every maximal affine interval $[r,s]$ and every $r\leq i\leq s$, combining the lower bound \eqref{eq:standard-exponential-lower-bound} and \cref{thm:affine-interval-growth} yields
\[
0\leq\underline\mu_i(f)\leq\overline\mu_i(f)\leq 2(i-r)(s-i).
\]
It follows readily from \eqref{eq:projective-degree-log-concavity} and \eqref{eq:kahler-degree-log-concavity} that $\Delta\underline\mu_i(f)\coloneqq 2\underline\mu_i(f)-\underline\mu_{i-1}(f)-\underline\mu_{i+1}(f)\geq0$ for every $r<i<s$.
Also, similar to the proof of \cref{lemma:two-variable-maximum-principle}, the same discrete Phragm\'en--Lindel\"of argument shows that
\begin{align*}
\Delta\overline\mu_i(f)\coloneqq 2\overline\mu_i(f)-\overline\mu_{i-1}(f)-\overline\mu_{i+1}(f)\leq4,
\end{align*}
which extends \cite[Theorem~1.2(2)]{HJ-degree-growth} from regular or holomorphic maps to rational or meromorphic maps.
This leaves the following finer questions; cf.~\cite[Conjecture~2 and the subsequent paragraph]{Silverman14} and \cite[Question~13.16]{BIJMST19}.

\begin{question}\label{question:polynomial-dyn-deg}
Assume that either \cref{setting:projective} or \cref{setting:kahler} holds, and let $1\leq i\leq d-1$.

\begin{enumerate}[label=\textnormal{(\arabic*)}]
\item \emph{Existence.}
Does the limit
\[
\mu_i(f)\coloneqq\lim_{n\to\infty}\frac{\log\deg_{i,\omega}(f^n)-n\log\lambda_i(f)}{\log n}
\]
exist?

\item \emph{Two-sided growth.}
If the limit exists, does one have
\[
\deg_{i,\omega}(f^n)\asymp n^{\mu_i(f)}\lambda_i(f)^n?
\]

\item \emph{Integrality.}
If the limit exists, is $\mu_i(f)\in\bZ_{\geq0}$?
\end{enumerate}
\end{question}

\begin{remark}\label{rmk:polynomial-dyn-deg-known}
Beyond the functorial case just discussed, we record several classes for which all three parts of \cref{question:polynomial-dyn-deg} are known.

\begin{enumerate}[label=\textnormal{(\arabic*)}]
\item For \emph{birational} or \emph{bimeromorphic} self-maps of surfaces, the classification of Diller--Favre \cite{DF01} completely settles all three parts: $\mu_1(f)=0$ if $\lambda_1(f)>1$, and $\mu_1(f)\in\{0,1,2\}$ if $\lambda_1(f)=1$; see also \cite[Remark~1.2]{Oguiso16}.

\item Favre--Jonsson \cite[Theorem~C]{FJ11} proved that if $f$ is a dominant polynomial self-map of $\bC^2$ satisfying $\lambda_2(f)=\lambda_1(f)^2$, then either the polynomial degree $\deg(f^n)\asymp\lambda_1(f)^n$ or $\deg(f^n)\asymp n\lambda_1(f)^n$; in particular, $\mu_1(f)\in\{0,1\}$.

\item Monomial maps $f\colon\bbP^d\dashrightarrow\bbP^d$ induced by matrices $M\in\Mat_d(\bZ)$ with $\det M\neq0$ form another class for which all three parts have affirmative answers in every codimension; see \cite[Proposition~4.4]{Lin12b} or \cite[Theorem~A]{FW12}.

\item At every strict dynamical-degree break $\lambda_i(f)>s_i(f)$, \cref{thm:strict-break-asymptotics} shows that for every $s_i(f)<\eta<\lambda_i(f)$,
\[
\deg_{i,\omega}(f^n)=C_{f,\omega,i}\,\lambda_i(f)^n+O(\eta^n).
\]
Thus all three parts hold with $\mu_i(f)=0$.
The surface case was proved by Boucksom--Favre--Jonsson \cite[Main Theorem]{BFJ08}, and the codimension-one case in characteristic zero by Dang--Favre \cite[Theorem~1]{DF21}.
\end{enumerate}

Consequently, the remaining general problem is confined to resonant indices
\[
\lambda_i(f)^2=\lambda_{i-1}(f)\lambda_{i+1}(f).
\]
\end{remark}


\newcommand{\etalchar}[1]{$^{#1}$}
\providecommand{\bysame}{\leavevmode\hbox to3em{\hrulefill}\thinspace}
\raggedbottom

\end{document}